\documentclass[a4paper,10pt]{article}
\usepackage{amsmath,amssymb,wasysym,amsfonts,amsthm,amsthm}
\usepackage{empheq,hyperref,url,enumerate}
\usepackage[a4paper,margin=1.0in]{geometry}
\newtheorem{proposition}{Proposition}
\numberwithin{equation}{section}
\numberwithin{proposition}{section}
\usepackage{color}
\usepackage[dvipsnames]{xcolor}
\usepackage{cite}
\usepackage{graphicx,epsfig,subfig}
\usepackage[british]{babel}
\usepackage{soul}

\title{INFERRING PARAMETERS OF PREY SWITCHING IN A PLANKTON 1 PREDATOR--2 PREY SYSTEM WITH A LINEAR PREFERENCE TRADEOFF}

\author{Sofia H. Piltz \thanks{Corresponding author, Department of Applied Mathematics and Computer Science, Technical University of Denmark, Asmussens all\'{e}, Bygning 303B, 2800 Kongens Lyngby, Denmark and Department of Mathematics, University of Michigan, 2074 East Hall, Ann Arbor, Michigan, 48109-1043, USA \mbox{(\href{mailto:piltz@umich.edu}{piltz@umich.edu})}.}
\and  Lauri Harhanen \thanks{Department of Applied Mathematics and Computer Science, Technical University of Denmark, Asmussens all\'{e}, Bygning 303B, 2800 Kongens Lyngby, Denmark \mbox{(\href{mailto:lauri.harhanen@alumni.aalto.fi}{lauri.harhanen@alumni.aalto.fi})}.} 
\and Mason A. Porter\thanks{Department of Mathematics, University of California, Los Angeles, Los Angeles, California 90095, USA; Oxford Centre for Industrial and Applied Mathematics, Mathematical Institute, University of Oxford, Andrew Wiles Building, Radcliffe Observatory Quarter, Woodstock Road, Oxford, OX2 6GG, UK; and CABDyN Complexity Centre, University of Oxford, Oxford, OX1 1HP, UK \mbox{(\href{mailto:mason@math.ucla.edu}{mason@math.ucla.edu})}.}
\and Philip K. Maini \thanks{Wolfson Centre for Mathematical Biology, Mathematical Institute, University of Oxford, Andrew Wiles Building, Radcliffe Observatory Quarter, Woodstock Road, Oxford, OX2 6GG, UK and CABDyN Complexity Centre, University of Oxford, Oxford, OX1 1HP, UK. \mbox{(\href{mailto:maini@maths.ox.ac.uk}{maini@maths.ox.ac.uk})}.}
}
\begin{document}
\maketitle

\begin{abstract}

We construct two ordinary-differential-equation models of a predator feeding adaptively on two prey types, and we evaluate the models' ability to fit data on freshwater plankton. We model the predator's switch from one prey to the other in two different ways: (1) smooth switching using a hyperbolic tangent function; and (2) by incorporating a parameter that changes abruptly across the switching boundary as a system variable that is coupled to the population dynamics. We conduct linear stability analyses, use approximate Bayesian computation (ABC) combined with a population Monte Carlo (PMC) method to fit model parameters, and compare model predictions quantitatively to data for ciliate predators and their two algal prey groups collected from Lake Constance on the German--Swiss--Austrian border. We show that the two models fit the data well when the smooth transition is steep, supporting the simplifying assumption of a discontinuous prey switching behavior for this scenario. We thus conclude that prey switching is a possible mechanistic explanation for the observed ciliate--algae dynamics in Lake Constance in spring, but that these data cannot distinguish between the details of prey switching that are encoded in these different models.
\end{abstract}

\paragraph{Keywords:} prey switching, Lotka--Volterra interactions, linear stability analysis, planktonic ciliate--algae dynamics, smoothening out a discontinuous diet switch, parameter fitting using PMC--ABC

\section{Introduction}\label{section_introduction} 

Ciliates are eukaryotic single cells that propel using small protuberances (called \emph{cilia}) that project from their cell body. They feed on small algae and are an important link between the bottom and higher levels of aquatic food webs \cite{TirokGaedke2007}. In addition to seasonal variation, ciliates and their algal prey populations vary at shorter-than-seasonal temporal scales. During years when the spring bloom lasts for several weeks (corresponding to approximately 15--30 ciliate generations), algal and ciliate biomasses exhibit recurring patterns of growth followed by decline \cite{TirokGaedke2007}. Ciliates have different modes of predatorial behavior, and they can be categorized roughly in terms of more-selective or less-selective feeding habits \cite{Verity1991}. Some ciliate species can be construed as selective predators, because they hunt as ``interception feeders" that scavenge on food particles and intercept them directly. In contrast, ``filter-feeder" ciliates sieve suspended food particles and are an example of a less-selective ciliate species. A laboratory experiment on ciliate predator and phytoplankton prey species 
%found 
in Lake Constance reported prey preference and selective feeding in ciliates \cite{MullerSchlegel1999}, and it has been suggested that predator--prey interactions between diverse predator and prey plankton communities are the driving force for the sub-seasonal temporal variability observed in ciliate--algal dynamics, especially during periods of the year in which environmental conditions are relatively stable \cite{TirokGaedke2010}. 

In the present paper, we aim to obtain biological insight into the sub-seasonal oscillations in ciliate populations during spring in Lake Constance and more generally into the ecological concept of \emph{prey switching} \cite{Murdoch1969}, in which predators express a preference% for
, e.g., for
more-abundant prey. To do this, we construct multiple modeling frameworks for a ciliate predator that adaptively changes its diet in response to changes in the abundances of its two prey.\footnote{``Prey switching'' in a system of 1 predator and 1 prey refers to a situation in which predation is low at low prey densities but saturates quickly at a large value when the prey is abundant. In such a scenario, one can model the predator--prey interaction using a Holling type-III functional response \cite{Holling1965,Gauseetal1936}.} Using ciliate--algae interactions in Lake Constance as an example, we focus on adaptive feeding of a predator group between its two different types of prey to investigate both the characteristics of prey switching (specifically, whether it is best described with a steep or gradual switching function) and whether it is justified to use a reduced modeling framework (specifically, a \emph{piecewise-smooth dynamical system}\footnote{Piecewise-smooth dynamical systems are a class of discontinuous systems that describe behavior using smooth dynamics of variables that alternate with abrupt events \cite{pw-sBook,pws-scholarpedia}.}) as a simplifying approximation of a smooth system. 

One can model prey switching with smooth dynamical systems by considering either density-dependent switching \cite{AbramsMatsuda2003} or density-independent switching \cite{Postetal2000}, or by using information on which prey type was last consumed \cite{vanLeeuwen2007,vanLeeuwen2013}. In contrast, a piecewise-smooth system arises when one assumes that a switch in a predator's feeding behavior depends on prey abundances. For example, one can posit that a predator behaves as an optimal forager, as its choice to switch prey depends on which diet composition maximizes its rate of energy intake \cite{Stephens1986,Krivan1996,KrivanSikder1999,BoukalKrivan1999,KrivanEisner2003, Krivan2007}. Using a piecewise-smooth model, it was suggested recently that prey switching is a possible mechanistic explanation for the dynamics observed in ciliate and algae populations in Lake Constance \cite{OurPaper}. 

In addition to their ecological applications, piecewise-smooth dynamical systems occur in a wide variety of applications \cite{pw-sBook}, ranging from mechanical oscillators such as a rocking block (see, e.g, \cite{hogan1989RockingBlockdynamics}) to relay-feedback systems (in which an electrically operated switch is used to control a process or an electromechanical system \cite{bernardo2001RelFBControlPWS}). Other biological applications include gene regulatory networks, in which transcription factors either initiate or inhibit the production of proteins after some threshold concentration has been reached \cite{Glass1975, Casey2006}, and conceptual climate models, in which an abrupt change in a piecewise-smooth system can represent a transition between different regimes in, e.g., large-scale ocean circulation \cite{stommel1961thermohaline} or the Earth's reflectivity due to ice cover \cite{abbot2011jormungand}. In a piecewise-smooth dynamical system, phase space is divided into two or more smooth regions by one or more \emph{switching manifolds} that mark transitions between the regions. For prey switching, each region corresponds to one of the predator's diet choices, so a model that describes the dynamics has a discontinuous right-hand side. Specifically, in this example, the system satisfies a different set of ordinary differential equations (ODEs) in different regions of phase space. In a piecewise-smooth framework, one assumes that a switch from one diet to another occurs instantaneously. Consequently, piecewise-smooth dynamical systems are also used to approximate nonlinear terms, such as sigmoidal or cubic functions, in models of systems that have sharp transitions between two or more states. 

In both ecology and numerous other applications, it is important to conduct detailed investigations into different approaches for how to model sharp changes in governing dynamics. On one hand, it is unclear whether there exist ``discontinuous predators" who switch their feeding strategy instantaneously, as assumed in a piecewise-smooth model for prey switching. On the other hand, we have not found evidence for any of the possible smooth transition functions that one can choose to model prey switching. By using data on ciliate and algal population dynamics, we illustrate an important example situation in which exploiting different modeling frameworks increases understanding of prey switching and allows one to gain biological insight into its profile. Consequently, we can conclude that it is justified to use a piecewise-smooth dynamical system, which has fewer parameters than the associated smooth models introduced in this paper, as a simplifying approximation of a smooth dynamical system. We also show that the piecewise-smooth model in \cite{OurPaper} is both biologically and mathematically consistent as the limit of two smooth systems, which we construct by (1) using a hyperbolic tangent as a transition function from one diet choice to another and (2) incorporating a parameter that changes abruptly across the discontinuity in the model in \cite{OurPaper} as a system variable with dynamics on a time scale comparable to that of the population dynamics of a predator and its two prey. In the second construction, we examine a system with one more dimension than the corresponding piecewise-smooth system.\footnote{Mathematically, one can derive different smooth approximations to the same piecewise-smooth system either by ``smoothing out'' a discontinuity of a piecewise-smooth system using a differentiable transition function of sigmoidal form \cite{Colombo2011} or by ``regularizing'' a piecewise-smooth dynamical system \cite{KuehnMTS2015} into a singular perturbation problem \cite{Hinch1991,Jones1995} that includes multiple time scales by ``blowing up'' the switching boundary \cite{Sotomayoretal1996,Teixeira2012}. In this work, we do not consider regularizations that include multiple time scales. A subset of us were among the authors of previous work \cite{ourfastslowPaper} on a multiple time-scale system describing the dynamics of one predator and two prey populations in the presence of rapid evolution of the predator's diet choice. One can also include nonlinear terms when constructing a smooth dynamical system by smoothing out an instantaneous switch by using the method developed in \cite{Jeffrey2014methodproposed,Jeffrey2014methodContinued}. These nonlinear terms take into account small effects that are observable only during a switch and vanish for the corresponding piecewise-smooth system \cite {JeffreySmoothingTautologies}. Comparing different smoothed out or regularized systems both with each other and with an associated piecewise-smooth system is crucial for understanding the correspondence and transition between a piecewise-smooth system and its smooth analogs. This type of relationship between piecewise-smooth and smooth systems that they approximate was investigated from a theoretical point of view in \cite{hogan2004effect,dankowicz2007purposeful}.}

The remainder of our paper is organized as follows. In Section \ref{section_Models}, we present and briefly discuss the equations for the 1 predator--2 prey piecewise-smooth model from \cite{OurPaper}. This piecewise-smooth dynamical system includes a tilted switching manifold that marks a transition between two smooth parts of phase space. Biologically, these parts represent the predator's adaptive feeding behavior and its two different diet choices: on one side of the switching manifold, the predator's diet consists solely of its preferred prey; on the other side, it consists solely of the alternative prey. We consider two possible regularizations of the model in Sections \ref{subsection_smoothAnalogI} and \ref{subsection_smoothAnalogII}, and we derive analytical expressions and carry out linear stability analysis for the coexistence equilibrium (i.e., where all three species coexist at non-zero densities) for each of the two smooth models. We are interested in the coexistence steady states, because the data that we use includes coexistence of predators and multiple prey. In Section \ref{section_dataComparison}, we discuss and use these data on adaptively feeding plankton predators to fit model parameters and compare the biomass predictions of our two smooth models. We summarize and discuss similarities and differences in model behavior and model assumptions between the piecewise-smooth system (analyzed in \cite{OurPaper}) and its two smooth analogs (analyzed in this paper) in Section \ref{section_discussion}, and we conclude our study in Section \ref{section_conclusions}. We give additional details about our calculations and analysis in a trio of appendices.

%%%%%

\section{Methods}\label{section_Models}

We construct two smooth analogs of a piecewise-smooth dynamical system describing a predator population $z$ that can adjust the extent of its consumption of its preferred prey $p_1$. When not consuming the preferred prey $p_1$, the predator feeds on its alternative prey $p_2$. Before introducing our two smooth models (see Sections \ref{subsection_smoothAnalogI} and \ref{subsection_smoothAnalogII}), we present the model equations for the piecewise-smooth dynamical system that was developed in \cite{OurPaper}. For each of the three models, we consider standard nonlinearities in the form of Lotka--Volterra predator--prey interactions. Although these nonlinearities are standard, it is convenient for us to use non-standard notation for the model coefficients that describe them. This notation allows us both to derive the switching condition introduced previously in \cite{OurPaper} and to compare the two smooth models that we develop in the present paper to this piecewise-smooth system.

We assume that the predator switches to consume only an alternative prey $p_2$ when it maximizes its fitness by doing so. To describe this situation, Ref.~\cite{OurPaper} developed the following piecewise-smooth dynamical system:
\begin{equation}
	\dot{{\mathbf x}}=
 	\left[\begin{array}{c} \dot{p_1}\\
 	\dot{p_2}\\ 
 	\dot{z}\\ \end{array}\right]=
 	\left\{ \begin{array}{rl}
		f_+= \left[\begin{array}{c}
		(r_1-\beta_1z)p_1\\
 			r_2p_2\\
 			(eq_1\beta_1p_1-m)z
 		\end{array}\right]\,,
& 		\mbox{ if $h=\beta_1p_1-a_q\beta_2p_2>0$} \\
 		f_-= \left[\begin{array}{c}
 		r_1p_1\\
 		(r_2-\beta_2z)p_2\\
 		(eq_2\beta_2p_2-m)z
 		\end{array} \right]\,,
& \mbox{ if $h=\beta_1p_1-a_q\beta_2p_2<0$}
 	\end{array} \right.\,,
\label{1pred2prey}
\end{equation}
where $r_1$ and $r_2$ (with $r_1>r_2>0$) are the respective per capita growth rates of the preferred and alternative prey, $\beta_1$ and $\beta_2$ are the respective death rates of the preferred and alternative prey due to predation, $m>0$ is the predator per capita death rate per day, and $e>0$ is the predator conversion efficiency. The parameters $q_1$ and $q_2$ are nondimensional parameters that represent the predator's respective desire to consume the preferred and alternative prey. Thus, the proportion of predation that goes into predator growth is given by $eq_1$ for the preferred prey and by $eq_2$ for the alternative prey. In other words, the model in Equation~\eqref{1pred2prey} preserves the idea of a predator's benefit from consuming prey being in proportion to predation amount, as is the case in the standard Lotka--Volterra model. Specifically, we consider constant conversion efficiency as a fraction of $e$ for each prey, so the benefit from feeding on the preferred and alternative prey are represented by $eq_1$ and $eq_2$, respectively. With $0<q_2<q_1<1$, we emphasize the reduced benefit that the predator obtains from the alternative prey compared to its preferred prey. 
The difference in the benefit is also where the assumed tradeoff lies, as the alternative prey $p_2$ invests energy in building predator defense mechanisms and is thus a ``less edible'' prey compared to the preferred prey $p_1$, which does not invest energy in predator defense mechanisms. Consequently, we assume that the growth rate of the preferred prey is larger than that of the alternative prey (i.e., $r_1>r_2$).  

In our constructions (see Sections \ref{subsection_smoothAnalogI} and \ref{subsection_smoothAnalogII}) of two smooth analogs of \eqref{1pred2prey} and to facilitate our comparison between the piecewise-smooth and smooth systems, we take $\beta_1=\beta_2=1$ for simplicity. In doing so, we assume that the predator exhibits adaptive feeding behavior by adjusting its preference (rather than its attack rate) to the governing prey densities. The parameter $a_q$ corresponds mathematically to the slope of the tilted switching manifold $h=\beta_1p_1-a_q\beta_2p_2=0$ between the two vector fields in \eqref{1pred2prey}. Biologically, $a_q$ is the slope of the assumed linear tradeoff in the predator's preference for prey. The reader is referred to \cite{OurPaper} for a biological justification of these model assumptions, analysis of the model \eqref{1pred2prey}, and inferred parameter values for data from Lake Constance. 

In the present paper, we construct and carry out linear stability analysis of two novel (to our knowledge) smooth models for an adaptively feeding predator and its two prey. First, in Section \ref{subsection_smoothAnalogI}, we formulate the model in \eqref{1pred2prey} as a three-dimensional smooth dynamical system with hyperbolic tangent functions. Second, in Section \ref{subsection_smoothAnalogII}, we construct a four-dimensional smooth analog of \eqref{1pred2prey} by supposing that the desire to consume the preferred prey $q_1$ changes between $1$ and $0$ across the discontinuity in the piecewise-smooth system \eqref{1pred2prey} as a system variable that changes along with the population dynamics. 

\subsection{Smooth model I}\label{subsection_smoothAnalogI}

We construct a smooth analog (which we call our ``smooth model I'') of \eqref{1pred2prey} using a hyperbolic tangent as a transition function. This yields the following equations of motion:
\begin{align}
	\dot{p_1} &= r_1p_1-\beta_1p_1z\left(\frac{1+\text{tanh}(k(\beta_1p_1-a_q\beta_2p_2))}{2}\right)
\equiv f(p_1,p_2,z)\,, \nonumber \\ 
	\dot{p_2} &= r_2p_2-\beta_2p_2z\left(\frac{1-\text{tanh}(k(\beta_1p_1-a_q\beta_2p_2))}{2}\right)\equiv g(p_1,p_2,z)\,, \label{smooth_1pred2prey}\\
	\dot{z} &= eq_1\beta_1p_1z\left(\frac{1+\text{tanh}(k(\beta_1p_1-a_q\beta_2p_2))}{2}\right)
+eq_2\beta_2p_2z\left(\frac{1-\text{tanh}(k(\beta_1p_1-a_q\beta_2p_2))}{2}\right)-mz \equiv l(p_1,p_2,z)\nonumber \,,
\end{align}
where $k$ determines the steepness of the transition function and thus of switches in the predator's feeding behavior. Equation~\eqref{smooth_1pred2prey} incorporates Lotka--Volterra dynamics, and $eq_1$ and $eq_2$ (where $0<q_2<q_1<1$) can be construed, respectively, as one predator's benefit from eating its preferred and alternative prey. In Section \ref{subsection_comparisonwDataSysI}, we will infer values of $k$ that best fit data from a particular freshwater plankton system. The data were collected in Lake Constance between 1979 and 1999, were presented originally in \cite{TirokGaedke2006,TirokGaedke2007} and were subsequently analyzed further in several papers (e.g., \cite{TirokGaedke2010,TirokGaedke2007a}). See Section \ref{section_dataComparison} for a description of the data.

%%%%%

\subsubsection{Linear stability analysis of smooth model I} 

In this work, we are interested in a steady state of \eqref{smooth_1pred2prey} with $p_1,p_2,z>0$. We calculate
\begin{align}
	f &= p_1\left(r_1-z\left(\frac{1+\text{tanh}(k(p_1-a_qp_2))}{2}\right)\right)=0 \nonumber \\
&\Rightarrow z\left(\frac{1+\text{tanh}(k(p_1-a_qp_2))}{2}\right)=r_1 \label{dotp_1_2ndStep}
\end{align}
and
\begin{align}
	g &= p_2\left(r_2-z\left(\frac{1-\text{tanh}(k(p_1-a_qp_2))}{2}\right)\right)=0 \nonumber \\
&\Rightarrow z\left(\frac{1-\text{tanh}(k(p_1-a_qp_2))}{2}\right)=r_2\,, \label{dotp_2_2ndStep}
\end{align}
so
\begin{align}
	l = (eq_1p_1-m)r_1+(eq_2p_2-m)r_2=0 \label{hEqualToZeroSubstituted}\,.
\end{align}

We obtain the steady-state solution $(\tilde{p}_1,\tilde{p}_2,\tilde{z})$, where
\begin{align}
	\tilde{z} &=r_1+r_2 \nonumber \\
	(eq_1\tilde{p}_1-m)r_1+(eq_2\tilde{p}_2-m)r_2&=0 \label{steadystateSmoothSys1}\\
\text{tanh}(k(\tilde{p}_1-a_q\tilde{p}_2))&=\frac{r_1-r_2}{r_1+r_2}\,. \nonumber
\end{align}
Taking the inverse hyperbolic tangent on both sides of the third equation in \eqref{steadystateSmoothSys1} results in linearly independent equations for $\tilde{p}_1$ and $\tilde{p}_2$. We thereby obtain a unique steady state at 
\begin{align}
 	\tilde{p}_1&=\frac{a_qm(r_1+r_2)+\frac{eq_2r_2\text{arctanh}\left(\frac{r_1-r_2}{r_1+r_2}\right)}{k}}{e(q_1a_qr_1+q_2r_2)} \,,\nonumber\\
 	\tilde{p}_2&=\frac{m(r_1+r_2)-\frac{eq_1r_1\text{arctanh}\left(\frac{r_1-r_2}{r_1+r_2}\right)}{k}}{e(q_1a_qr_1+q_2r_2)}\label{coexistenceStStSmoothSysI} \,, \\
 	\tilde{z}&=r_1+r_2\nonumber\,.
\end{align}
All three population densities are positive at the steady state $(\tilde{p}_1,\tilde{p}_2,\tilde{z})$ when
\begin{equation}
 	k>k_0=\frac{eq_1r_1\text{arctanh}\left(\frac{r_1-r_2}{r_1+r_2}\right)}{m(r_1+r_2)}\,.
 \label{equationFork0}
\end{equation}
We use the Routh--Hurwitz criterion \cite{routh1877,hurwitz1895} to investigate the stability of the coexistence steady state \eqref{coexistenceStStSmoothSysI}.

\vspace{.5cm}

\begin{proposition}
 \,\,If $a_q \geq q_2/q_1$, then the steady state \eqref{coexistenceStStSmoothSysI} is stable if and only if $k > k_0$.
 \label{propositionSmoothModelI_1}
\end{proposition}

\begin{proof}
See Appendix \ref{appendixStabilityI}.
\end{proof}

\vspace{.5cm}

\begin{proposition}
\,\,If $a_q < q_2/q_1$, then there exists $k_1 \in (k_0,\infty)$ so that the steady state \eqref{coexistenceStStSmoothSysI} is stable if and only if $k \in (k_0, k_1)$.
\label{propositionSmoothModelI_2}
  \end{proposition}

\begin{proof}
See Appendix \ref{appendixStabilityI}.
\end{proof}

\vspace{.5cm}

From these results, we see that when $a_q$ is large, which corresponds to a predator with a steep tradeoff in its prey preference (i.e., a small increase in specialization towards the preferred prey comes at a large cost in growth from feeding on the alternative prey), the coexistence steady state \eqref{coexistenceStStSmoothSysI} is stable for all $k>k_0$. When the prey switching is steep (i.e., when $k \rightarrow \infty$), the coexistence steady state \eqref{coexistenceStStSmoothSysI} is equal to the steady state of the piecewise-smooth system that lies on the switching manifold (i.e., it is a \emph{pseudoequilibrium}), and it has a complex-conjugate pair of eigenvalues with negative real part when $a_q>q_2/q_1$ \cite{OurPaper}. However, in contrast to the piecewise-smooth system, in which the coexistence steady state is repelling for shallow or flat prey preference tradeoffs (i.e., when $a_q<q_2/q_1$), the smooth system \eqref{smooth_1pred2prey} has an interval of intermediate prey-switching slopes $k \in (k_0,k_1)$ (see Equation~\eqref{expressionOfK1} for the expression for $k_1$) for which the coexistence state is also stable for $a_q<q_2/q_1$. 

For population densities at the stable coexistence steady state, smooth model I predicts that the predator density is determined solely by the prey growth rates, so it is affected neither by the slope of the tradeoff nor by the steepness of the diet switch (see Equation \eqref{coexistenceStStSmoothSysI}). For a nearly flat tradeoff (i.e., when $a_q$ is small), the stable coexistence steady-state solution for the preferred prey $p_1$ is at its minimum (when other parameter values are as given in the caption of Figure \ref{coexistenceStStSmoothSysI_numPreyDensities}). However, if in addition to a mild tradeoff, the predator's prey switching is also gradual (i.e., $k$ is very small), then $p_1$ has a high concentration at the stable equilibrium (see the left panel of Figure \ref{coexistenceStStSmoothSysI_numPreyDensities}). This pattern of minimum and maximum values is reversed for the alternative prey: The steady-state concentration of the alternative prey $p_2$ has a large value when $a_q$ is small, except for very small $k$, when the steady-state value of $p_2$ is small (see the right panel of Figure \ref{coexistenceStStSmoothSysI_numPreyDensities}). Such behavior of the steady state \eqref{steadystateSmoothSys1} suggests that $k\rightarrow 0$ is a singular limit of smooth system I \eqref{smooth_1pred2prey}.

\begin{figure}[!] 
\centering
\epsfig{file=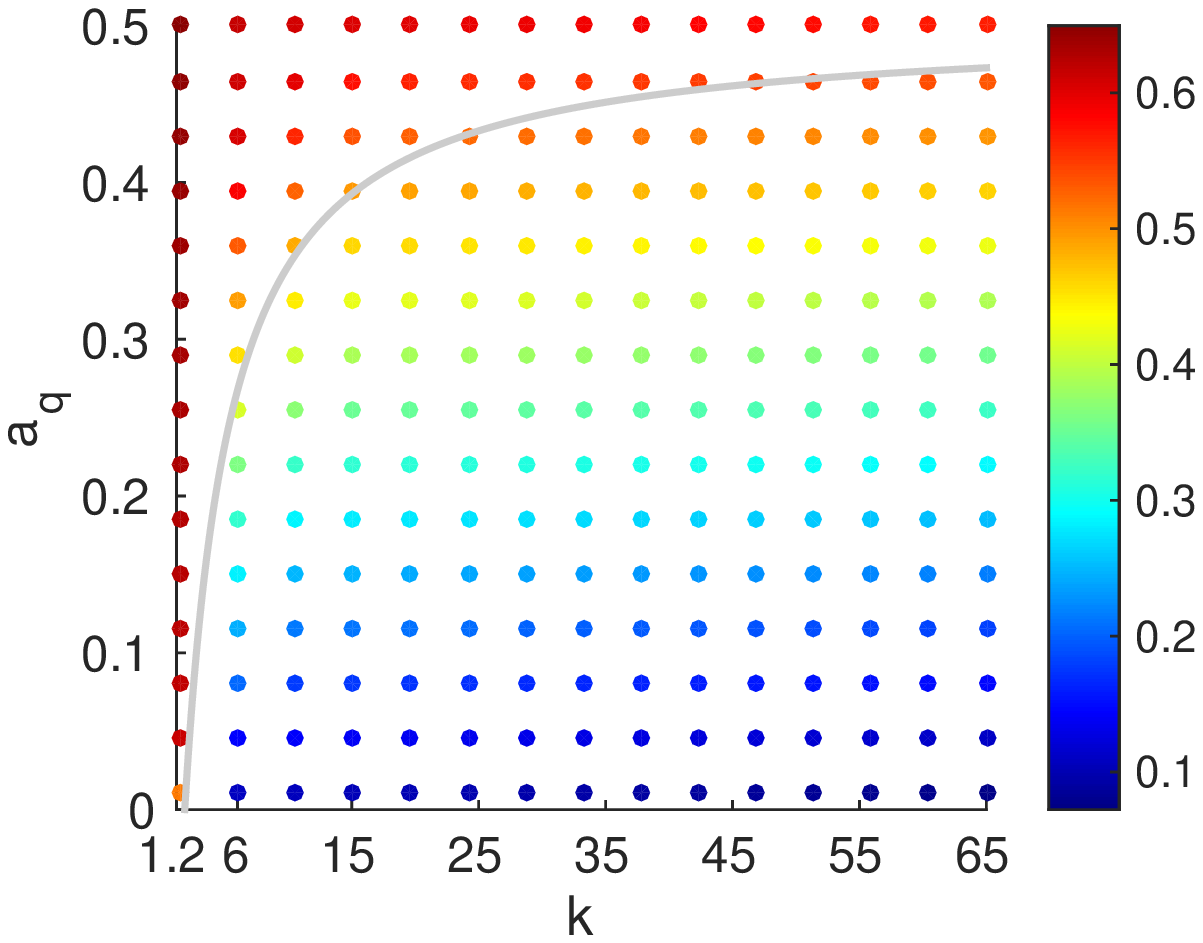,width=0.49\textwidth} 
\epsfig{file=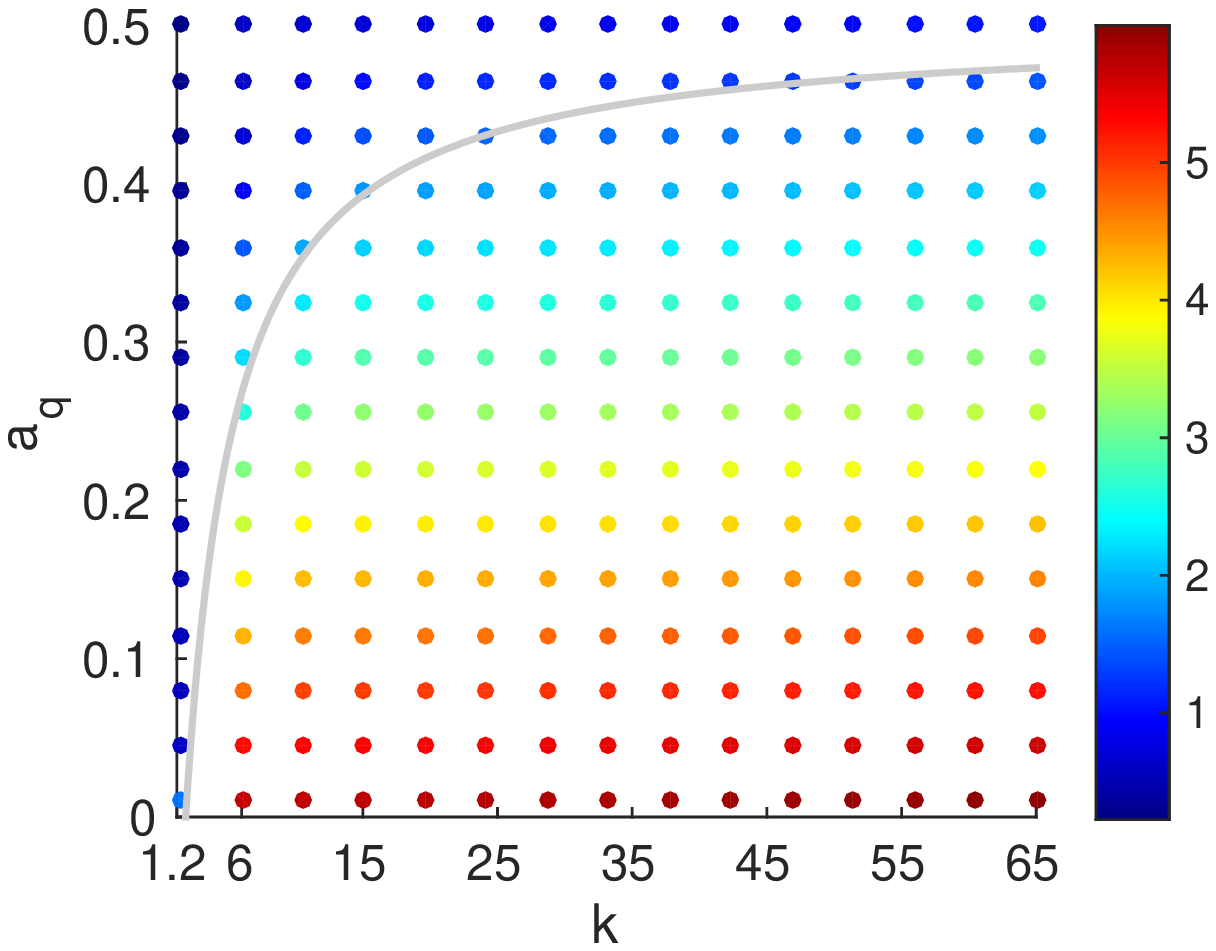,width=0.49\textwidth} 
\caption{Numerical computations for prey population densities at equilibrium in Equation \eqref{coexistenceStStSmoothSysI} for
parameter values $(e,\beta_1,\beta_2,r_1,r_2,m,q_1,q_2)=(0.25,1,1,1.3,0.26,0.14,1,0.5)$ [where we determine the values of $r_1$ ,$r_2$, and $m$ from our parameter fitting of the piecewise-smooth model in \cite{OurPaper}] of (left) the preferred prey $p_1$ and (right) the alternative prey $p_2$ at the indicated values of the slope of the preference tradeoff $a_q$ (vertical axis) and steepness of the predator switching $k$ (horizontal axis). The predator population density at equilibrium is $\tilde{z}=r_1+r_2 \approx 1.56$. We indicate the value of the prey density at equilibrium in color and numerically compute the steady-state solution of smooth system I \eqref{coexistenceStStSmoothSysI}. The steady state is stable above the gray curve. [See the equation for $k_1$ in Equation \eqref{expressionOfK1}.] With these parameter values, $k_0$ in Equation \eqref{equationFork0} is approximately $1.197$.} 
\label{coexistenceStStSmoothSysI_numPreyDensities} 
\end{figure}

%%%%%

\subsection{Smooth model II}\label{subsection_smoothAnalogII}

To regularize the three-dimensional piecewise-smooth system \eqref{1pred2prey} into a four-dimensional smooth system, we construct expressions for the temporal evolution of the predator's trait to accompany the population dynamics of the predator and two prey. Biologically, we are assuming that the predator's desire to consume its preferred prey undergoes either \emph{rapid evolution} \cite{Fussman2007} or \emph{phenotypic plasticity} \cite{Kellyetal_phenplast2012}, which are the two main forms of adaptivity in organisms. We will comment on these model assumptions in Section \ref{section_discussion}. We thereby turn the parameter $q_1$ that abruptly changes across the discontinuity in the piecewise-smooth model \eqref{1pred2prey} (i.e., $q_1=1$ when $h>0$ and $q_1=0$ when $h<0$) into a system variable $q$ that changes in response to prey abundance on the same time scale as the population dynamics in a smooth dynamical system. 

To ensure similarity with the piecewise-smooth model \eqref{1pred2prey}, we assume that no preference towards the preferred prey amounts to a feeding mode of consuming only the alternative prey (i.e., $q=0$) and that maximum preference towards the preferred prey amounts to a feeding mode of consuming only the preferred prey (i.e., $q=1$). We incorporate this assumption with a bounding function $q(1-q)$ in the expression for the temporal evolution of the predator's trait. From the condition for prey switching derived using optimal foraging theory \cite{Stephens1986} in \cite{OurPaper}, we impose that the rate of change of the mean trait value is proportional to $p_1-a_qp_2$. That is, we assume that the predator's choice to switch prey depends on prey abundances and which diet composition maximizes its rate of energy intake \cite{Stephens1986}. For simplicity, we also assume exponential prey growth and linear functional response as in the piecewise-smooth system \eqref{1pred2prey} \cite{OurPaper}. We thereby obtain the following dynamical system for the population dynamics coupled with temporal evolution of the predator trait: 
\begin{align}
	\frac{\text{d} p_1}{\text{d} t}&=g_1(p_1,p_2,z,q)=r_1p_1-qp_1z \nonumber \,,\\
	\frac{\text{d} p_2}{\text{d} t}&=g_2(p_1,p_2,z,q)=r_2p_2-(1-q)p_2z  \label{smoothSys2FullSystem} \,,\\ \nonumber
	\frac{\text{d} z}{\text{d} t}&=g_3(p_1,p_2,z,q)=eqp_1z+e(1-q)q_2p_2z-mz \nonumber \,,\\
	\frac{\text{d} q}{\text{d} t}&=f(p_1,p_2,q)=q(1-q)(p_1-a_qp_2)\,. \nonumber
\end{align} 
Similar to the piecewise-smooth system \eqref{1pred2prey} and to smooth system I \eqref{smooth_1pred2prey}, the predator--prey interaction in \eqref{smoothSys2FullSystem} (which we call ``smooth model II'') is of standard Lotka--Volterra type, so the benefit of consuming prey is proportional to the amount of predation. Consequently, the proportion of predation that goes into predator growth is given by $eq$ for the preferred prey and by $e(1-q)q_2$ for the alternative prey, where $e \in (0,1)$ is a parameter that represents conversion efficiency. For $q_1=1$ in the piecewise-smooth system \eqref{1pred2prey}, smooth model II \eqref{smoothSys2FullSystem} reduces to $f_+$ in \eqref{1pred2prey} when $q=1$ and to $f_-$ in \eqref{1pred2prey} when $q=0$. Biologically, these two cases correspond, respectively, to the situations in which the predator's diet is composed solely of the preferred prey and alternative prey. Note that the model in (\ref{smoothSys2FullSystem}) does not include a time-scale difference, which was incorporated between demographic and trait (i.e., $q_1$) dynamics of the four-dimensional smooth system, and analyzed using singular perturbation theory in \cite{ourfastslowPaper}.  

%%%%%%

\subsection{Linear stability analysis of smooth model II}\label{analysis_of_thesmoothSys2} 

The population densities of the predator and two prey at the coexistence steady state in smooth system II \eqref{smoothSys2FullSystem} are
\begin{align}
 	\tilde{p}_1 &=\frac{a_qm(r_1+r_2)}{e(r_1a_q+r_2q_2)} \,,\nonumber\\
 	\tilde{p}_2 &=\frac{m(r_1+r_2)}{e(r_1a_q+r_2q_2)}\,, \nonumber\\
 	\tilde{z} &=r_1+r_2 \,, \label{coexistStStSSII}\\
 	\tilde{q} &=\frac{r_1}{r_1+r_2}\,.\nonumber
\end{align}
These are the same densities that occur both at the steady state of smooth system I \eqref{coexistenceStStSmoothSysI} with steep prey switching (i.e., when $k \rightarrow \infty$) and at the pseudoequilibrium point of the piecewise-smooth system \eqref{1pred2prey} (for $q_1=1$ and $q_2=0.5$) that is located on the discontinuity boundary of the piecewise-smooth 1 predator--2 prey model \cite{OurPaper}. 

We summarize the result of linear stability analysis of smooth system II \eqref{smoothSys2FullSystem} in the two following propositions.

\vspace{.5 cm}

\begin{proposition}
 \,\,If $a_q = q_2$, then all eigenvalues of the steady state are purely imaginary.
\label{propositionSmoothModelII_1}
\end{proposition}
\begin{proof}
See Appendix \ref{appendixStabilityII}.
\end{proof}

\vspace{.5 cm}

\begin{proposition}
\,\,If $a_q \neq q_2$, then the steady state is linearly unstable. 
\label{propositionSmoothModelII_2}
\end{proposition}

\begin{proof}
See Appendix \ref{appendixStabilityII}.
\end{proof}

\vspace{.5 cm}

Consequently, smooth system II \eqref{smoothSys2FullSystem} has an unstable coexistence steady state irrespective of whether a predator can be construed as selective with a steep preference tradeoff with respect to its preferred and alternate prey or as unselective with a mild tradeoff in its preference towards the two prey. Our results also imply that our smoothing of the piecewise-smooth system \eqref{1pred2prey} by adding an extra dimension as in Equation \eqref{smoothSys2FullSystem} changes the stability of the coexistence equilibrium.

%%%%%

\section{Results}\label{section_dataComparison}

To obtain insight into the steepness of the prey switching in the two smooth models that we constructed in Section \ref{section_Models}, we consider data from Lake Constance (see \cite{TirokGaedke2006,TirokGaedke2007a}) for ciliate predators and two different types of their algal prey groups. The Lake Constance data set consists of over 23,000 observations of abundances (expressed either as individuals or as cells per milliliter) and biomass (expressed as units of carbon per square meter) of various plankton species obtained at least once in a sample of a few milliliters to a liter of water between March 1979 and December 1999. We compare the abundances predicted by our two smooth models with data from years 1991 and 1998. (For a comparison between the piecewise-smooth model \eqref{1pred2prey} and data, see \cite{OurPaper}.) During these two years, the spring bloom lasted for several weeks, and ciliate and algal biomasses exhibited recurring patterns of increases followed by declines \cite{TirokGaedke2007}. We are interested in spring abundances, because previous studies have suggested that predator--prey feeding interactions are an important factor in explaining the ciliate--algae dynamics in that season \cite{TirokGaedke2010} and that such interactions are more important than environmental conditions during spring \cite{Sommeretal2012}.

M{\"u}ller and Schlegel observed ciliates actively selecting against certain types of prey when offered a mixed diet of different types of their algal prey \cite{MullerSchlegel1999}. They suggested that adaptive feeding in ciliates occurs because different species benefit differently depending on the match between their feeding mode and the species that are abundant in the prey community. That is, ciliates select against their less-edible prey (e.g., a prey type that develops a hard silicate cover as a predator defense mechanism) when offered a mixed diet of both easily-digested and less-edible prey \cite{MullerSchlegel1999}. As a representative of an easily-digested prey group (i.e., preferred prey $p_1$ in the model equations), we consider data for {\em Cryptomonas ovata}, {\em Cryptomonas marssonii}, {\em Cryptomonas reflexa}, {\em Cryptomonas erosa}, {\em Rhodomonas lens}, and {\em Rhodomonas minuta} in the Lake Constance data set. For the less-edible prey (i.e., the alternative prey group $p_2$), we use data for small and medium-sized {\em Chlamydomonas} spp. and {\em Stephanodiscus parvus}. In addition to different prey groups, one can categorize ciliate-predators, which dominate the herbivorous zooplankton community in spring \cite{TirokGaedke2007}, roughly in terms of being more-selective or less-selective predators \cite{Verity1991}. To represent differences in selectivity between different predator species, the unselective filter-feeder predator group consists of data for {\em Rimostrombidum lacustris}, and the selective interception-feeder predator group consists of data for {\em Balanion planctonicum}. 

In this section, we use Lake Constance data on ciliate-predators and their algal prey to infer the steepness $k$ of the prey-switching function in smooth model I \eqref{smooth_1pred2prey}, the perturbation $\nu$ in the predator population from the coexistence steady state in Equation \eqref{coexistStStSSII} of smooth model II \eqref{smoothSys2FullSystem}, prey growth, predator death rates, and other parameters of our models. For our comparison between the Lake Constance data and the two smooth models in Sections \ref{subsection_comparisonwDataSysI} and \ref{subsection_comparisonwDataSysII}, we first normalize both the biweekly data points and the model predictions for the predator density, $z$, by their $L_2$-norm (i.e., by Euclidean distance). We consider the time window from 1 March to 15 June, for which there are 31 data points for the selective predator and 19 data points for the unselective predator in 1991. In 1998, there are 15 data points for both the selective and unselective predator species between 1 March and 15 June. We fit parameters to data with approximate Bayesian computation (ABC) combined with a population Monte Carlo (PMC) method \cite{beaumont2009PMCABC}. This combination allows us to study the results from the posterior parameter distribution rather than just a single value that gives the best fit as a result of an optimization method. Additionally, this approach allows us to code every step of the algorithm on our own. (We implement the algorithm in {\sc{Matlab}} \cite{MATLAB2014}.) We can thereby examine possible sources of error in the fitting process more easily than if providing input and analyzing an output of an available program package for parameter estimation of ODE models. The posterior parameter distribution, which is an output of the fitting algorithm, is especially useful for assessing how well the piecewise-smooth model \eqref{1pred2prey} approximates prey switching, which we represent with a hyperbolic tangent function in smooth model I \eqref{smooth_1pred2prey} and by incorporating an additional system variable in smooth model II \eqref{smoothSys2FullSystem}. 

%%%%

\subsection{Comparison of smooth model I simulations and Lake Constance data}\label{subsection_comparisonwDataSysI}

In this section, we fit the growth rates ($r_1$ and $r_2$, respectively) of the preferred and alternative prey, the predator mortality rate, $m$, the slope $k$ of the prey-switching function, and the slope $a_q$ of the prey-preference tradeoff of smooth model I \eqref{smooth_1pred2prey}. We use $a_q$ as a bifurcation parameter. (See Propositions \ref{propositionSmoothModelI_1} and \ref{propositionSmoothModelI_2}.) However, for simplicity (and similar to the study of the piecewise-smooth system in \cite{OurPaper}), we assume that the nondimensional preference parameters are fixed (and we take $q_1=1$ and $q_2=0.5$). Thus, given our choice of the preference parameters and using $a_q$ as a bifurcation parameter, we investigate linear preference tradeoffs that all go through point $(q_1,q_2)=(1,0.5)$ but do so with different slopes.

Smooth model I \eqref{smooth_1pred2prey} reproduces the peak abundances in the Lake Constance data and predicts an oscillatory pattern for both the selective and unselective predator populations during the springs of 1991 and 1998 (see Figures \ref{simulationAndDataSmoothSys1_1991} and \ref{simulationAndDataSmoothSys1_1998}). Additionally, our parameter fitting suggests that adaptive feeding of the selective predator is best represented with a steep switching function. In particular, for 1998, we obtain more frequent gradual prey-switching functions for an unselective predator than for a selective one at the smallest tolerance level of the fitting algorithm. See the middle rows of Figures \ref{simulationAndDataSmoothSys1_1991} and \ref{simulationAndDataSmoothSys1_1998}. Note that the coexistence steady state is unstable for the inferred parameter values that we use in the top-right panels in Figures \ref{simulationAndDataSmoothSys1_1991} and \ref{simulationAndDataSmoothSys1_1998}. In these two figures, this steady state is thus unstable for $k>k_1\approx1.3$ and $k>k_1\approx2.9$, respectively.

\begin{figure}[!] 
\centering
\epsfig{file=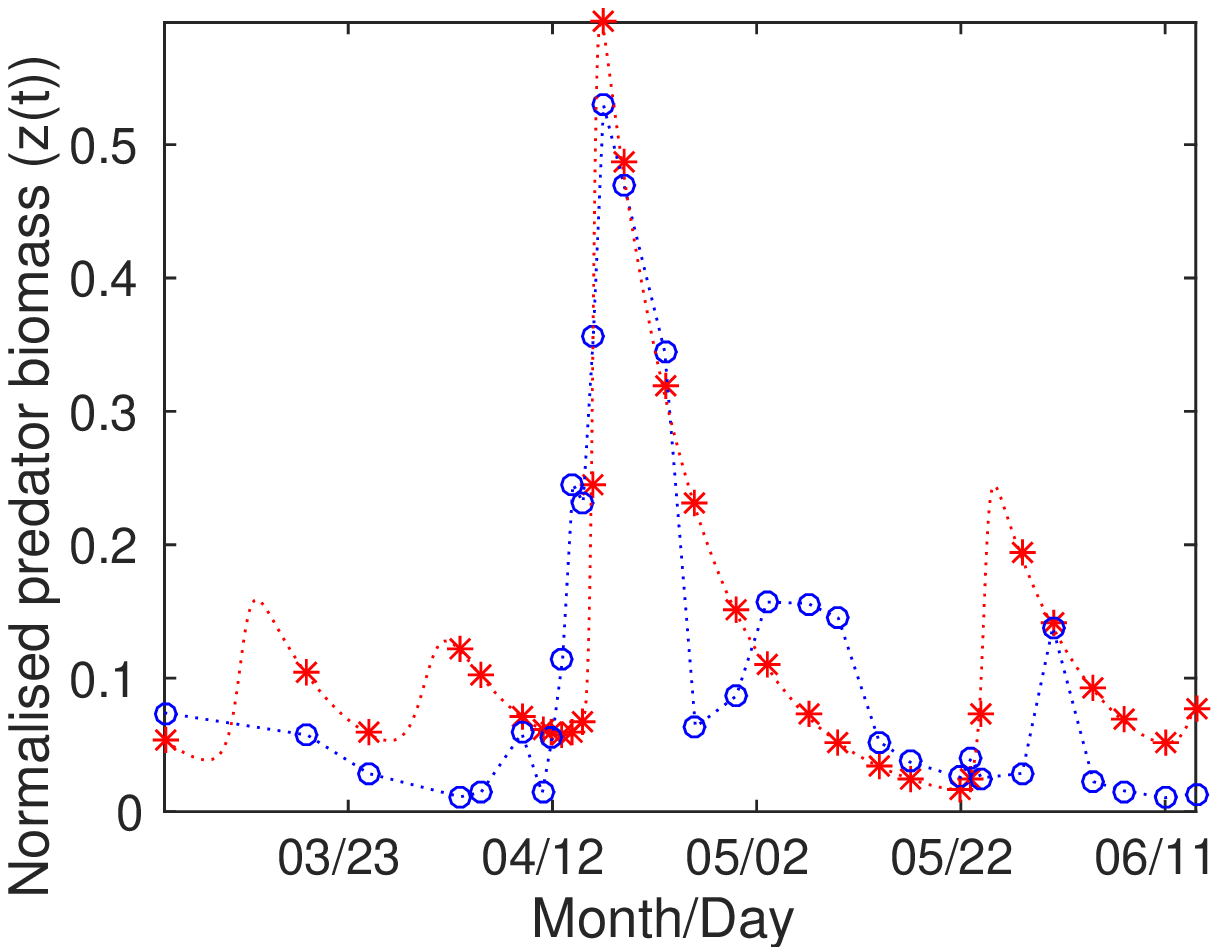,width=0.49\textwidth} 
\epsfig{file=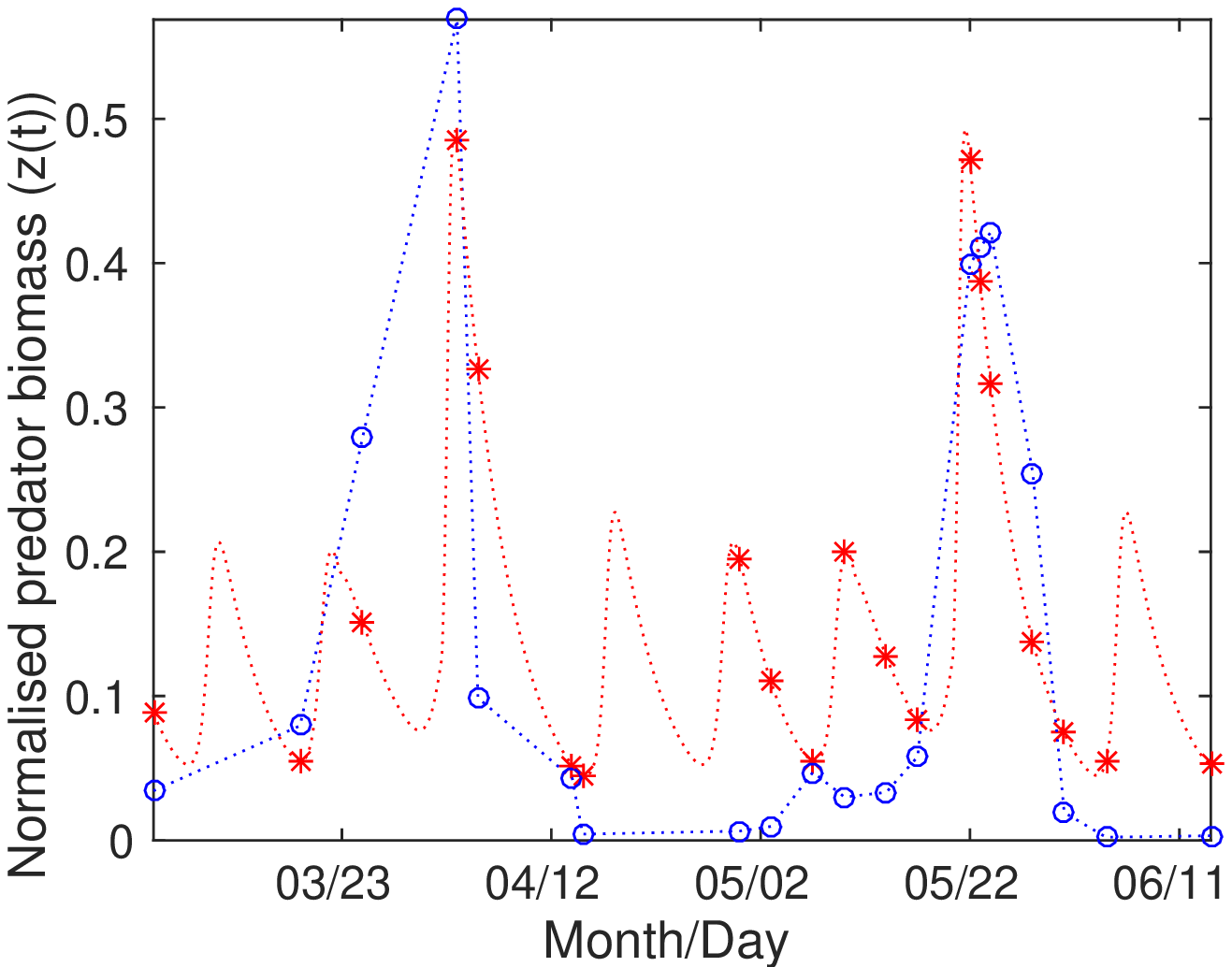,width=0.49\textwidth} 
\epsfig{file=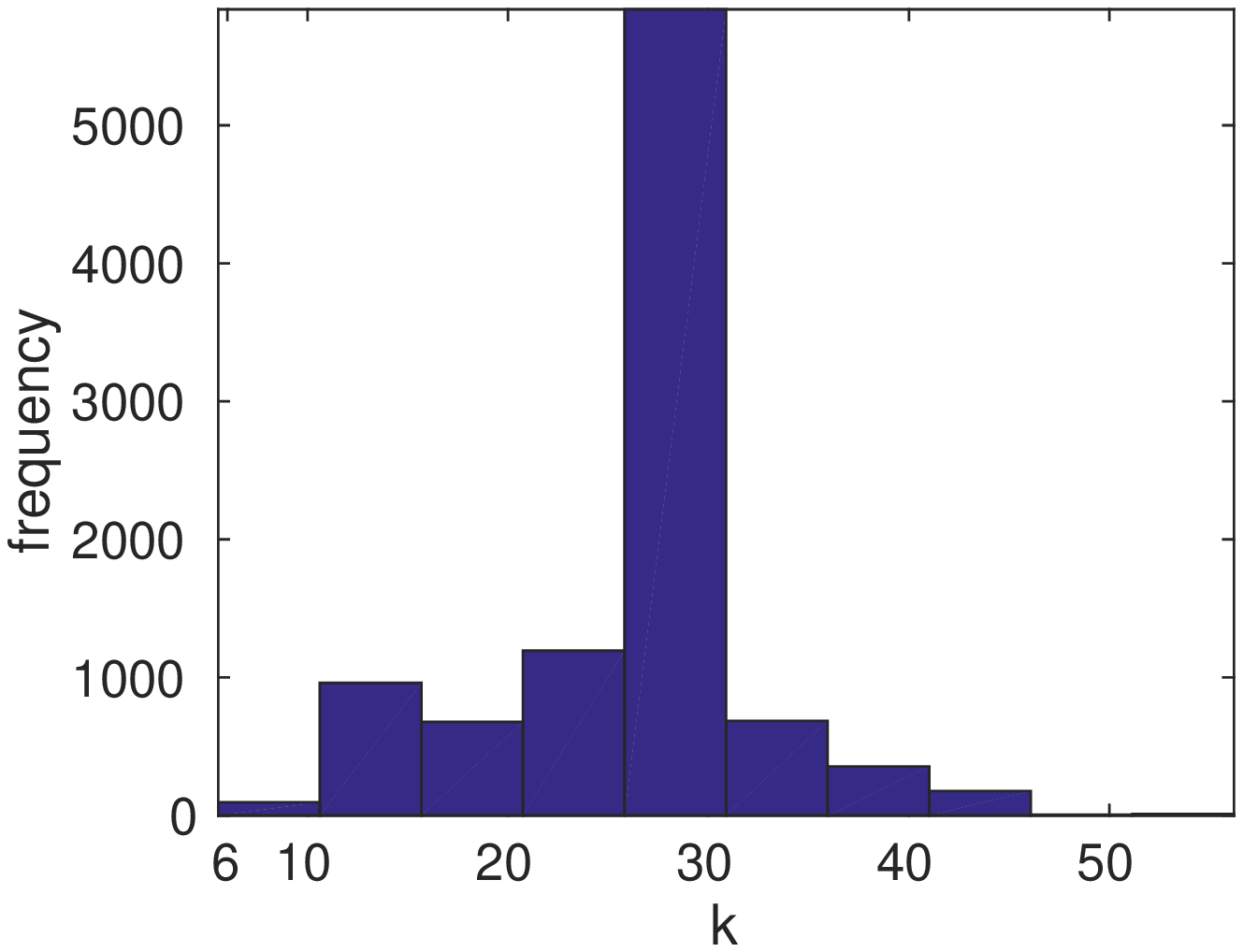,width=0.49\textwidth}
\epsfig{file=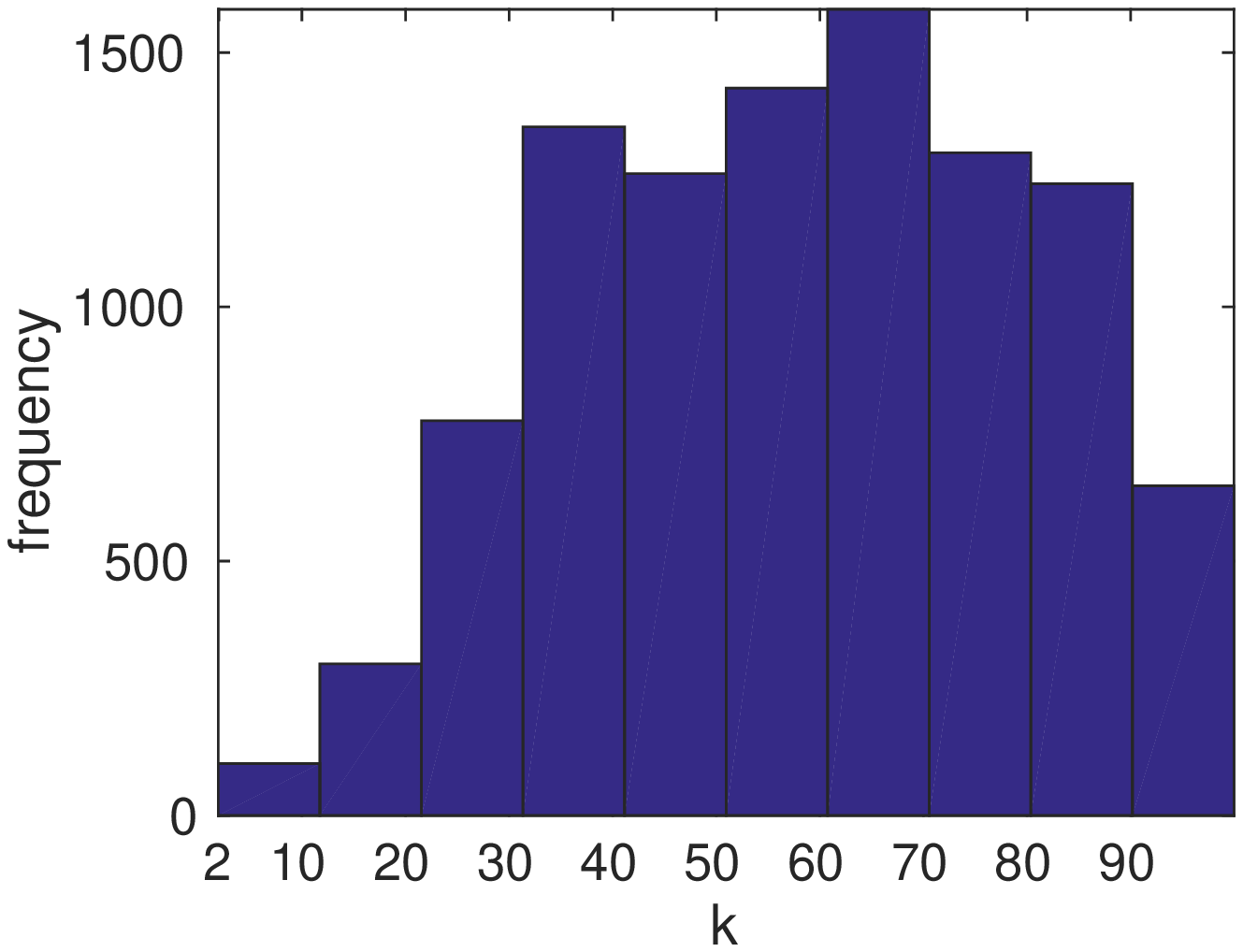,width=0.49\textwidth} 
\epsfig{file=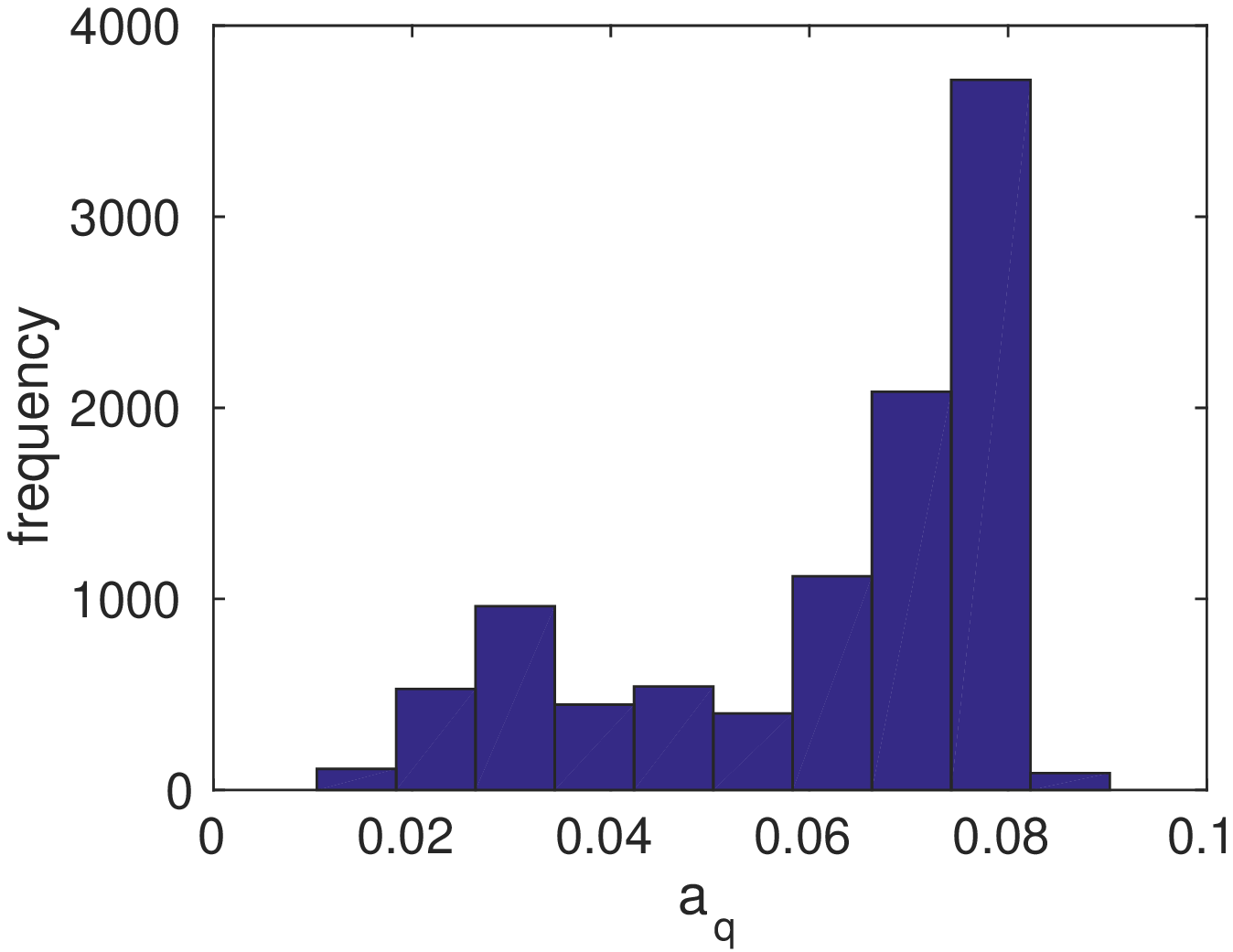,width=0.49\textwidth} 
\epsfig{file=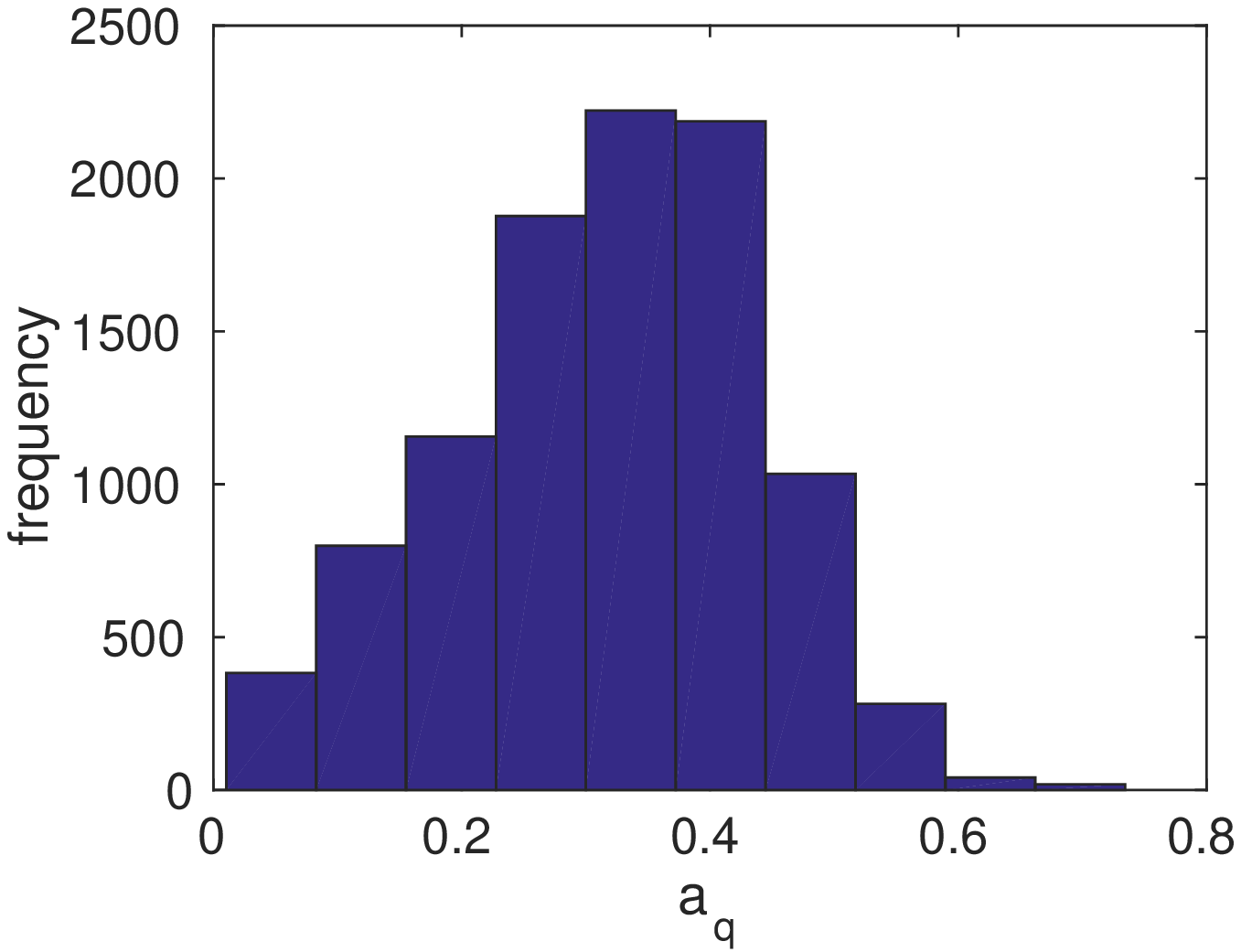,width=0.49\textwidth} 
\caption{(Top panels) The red asterisks give the normalized predator abundance $z(t)$ for simulations of smooth model I \eqref{smooth_1pred2prey} using the parameter values $q_1=1$, $q_2=0.5$, $e=0.25$, and $\beta_1=\beta_2=1$ and fitted values of (left) $r_1\approx1.64$, $r_2\approx0.62$, $m\approx0.11$, $a_q\approx0.02$, and $k\approx31$ and (right) $r_1\approx2.54$, $r_2\approx0.61$, $m\approx0.21$, $a_q\approx0.04$, and $k\approx67$. To guide the eye, we show the simulation in red between the asterisks. We show the normalized data using blue circles, and we show blue lines between them to guide the eye. We also show bar plots for the frequency of (center panels) $k$ values and (bottom panels) $a_q$ values at the strictest tolerance level ($Tol_{10}\approx0.00789$ in the left panels and $Tol_{15}\approx0.0258$ in the right panels) using the PMC ABC method \cite{beaumont2009PMCABC} for (left) selective and (right) unselective predator groups in spring in Lake Constance in 1991. Each frequency plot represents a random weighted sample (of size 10000) from the PMC ABC's posterior distribution of the parameter values accepted at the strictest tolerance level. The squared distances [see Equation \eqref{squaredDistance}] between the asterisks (model prediction) and circles (data) are (left) 0.0094 and (right) 0.0102. For more details on parameter fitting, see Appendix \ref{appendixParamFitting}. The unselective predator group consists of data for {\em Rimostrombidum lacustris}, and the selective predator group consists of data for {\em Balanion planctonicum}.
} 
\label{simulationAndDataSmoothSys1_1991} 
\end{figure}

\begin{figure}[!] 
\centering
\epsfig{file=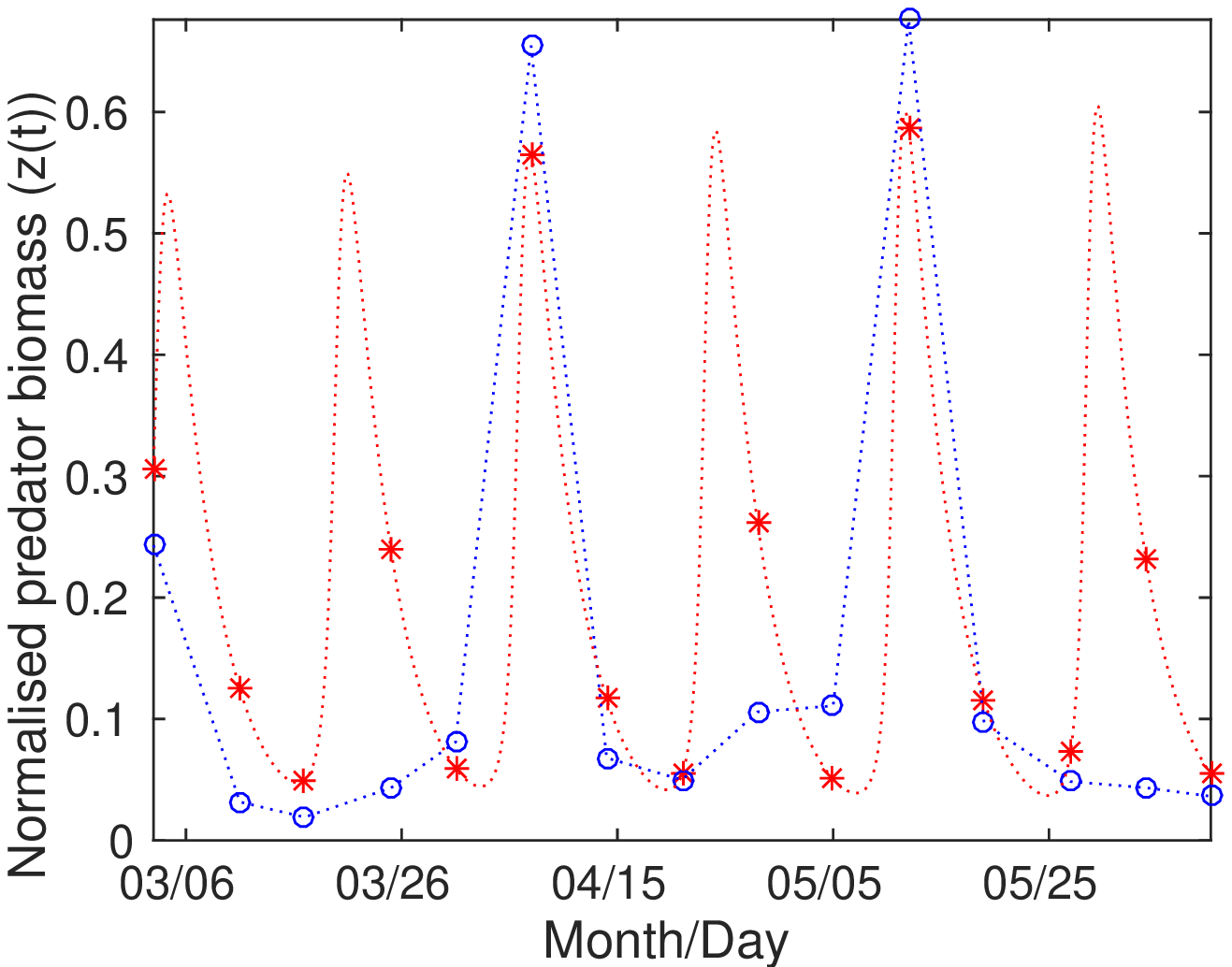,width=0.49\textwidth} 
\epsfig{file=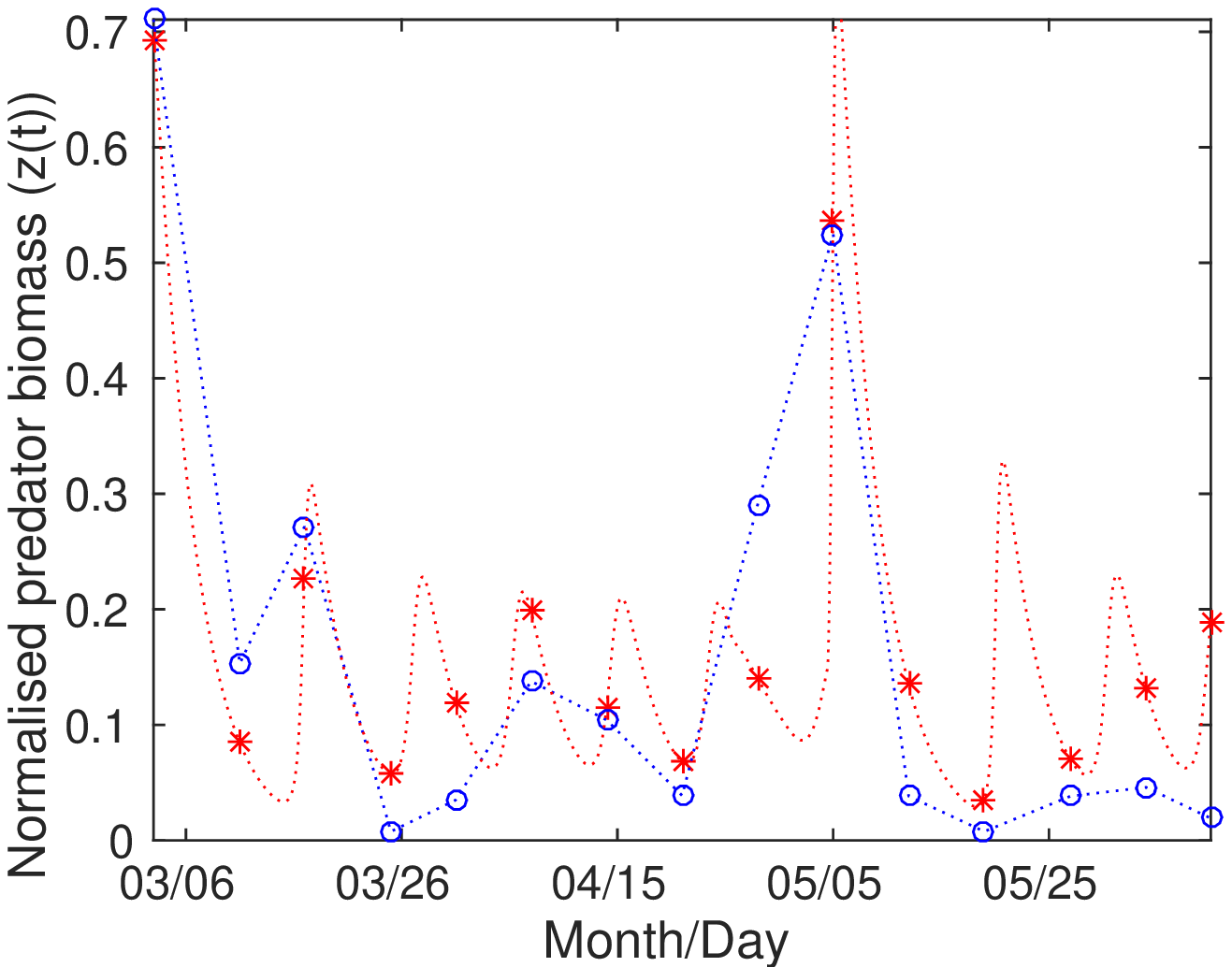,width=0.49\textwidth} 
\epsfig{file=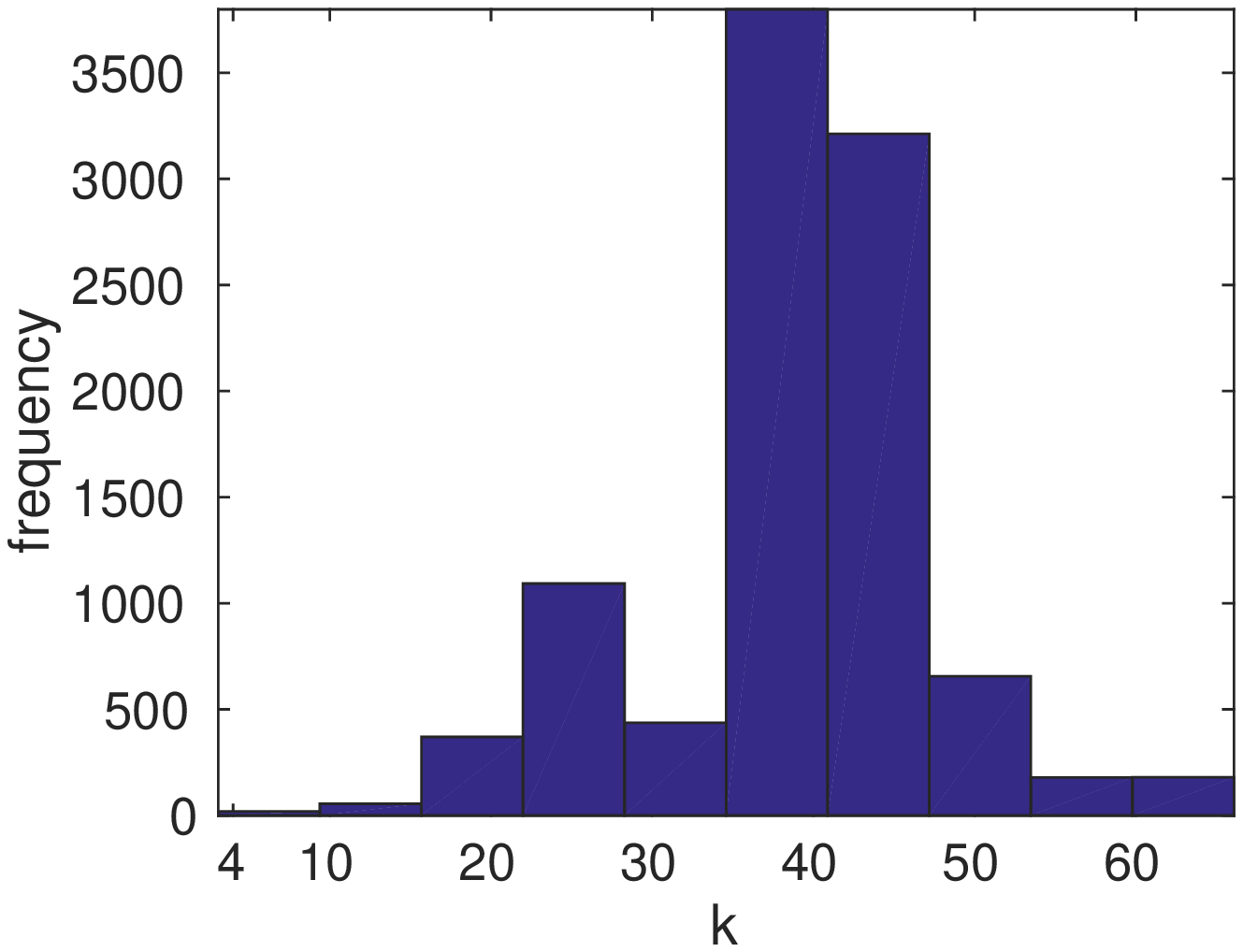,width=0.49\textwidth} 
\epsfig{file=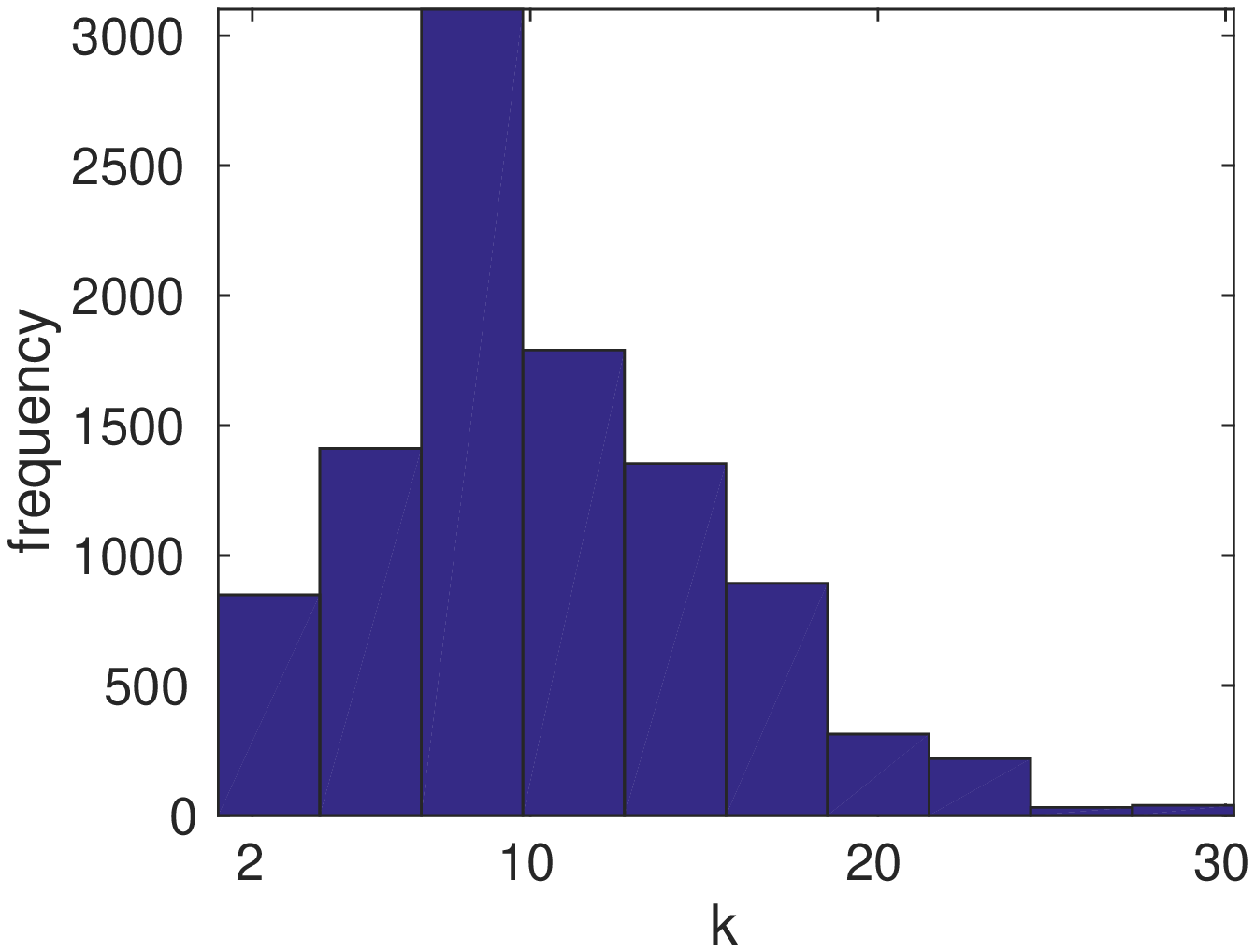,width=0.49\textwidth} 
\epsfig{file=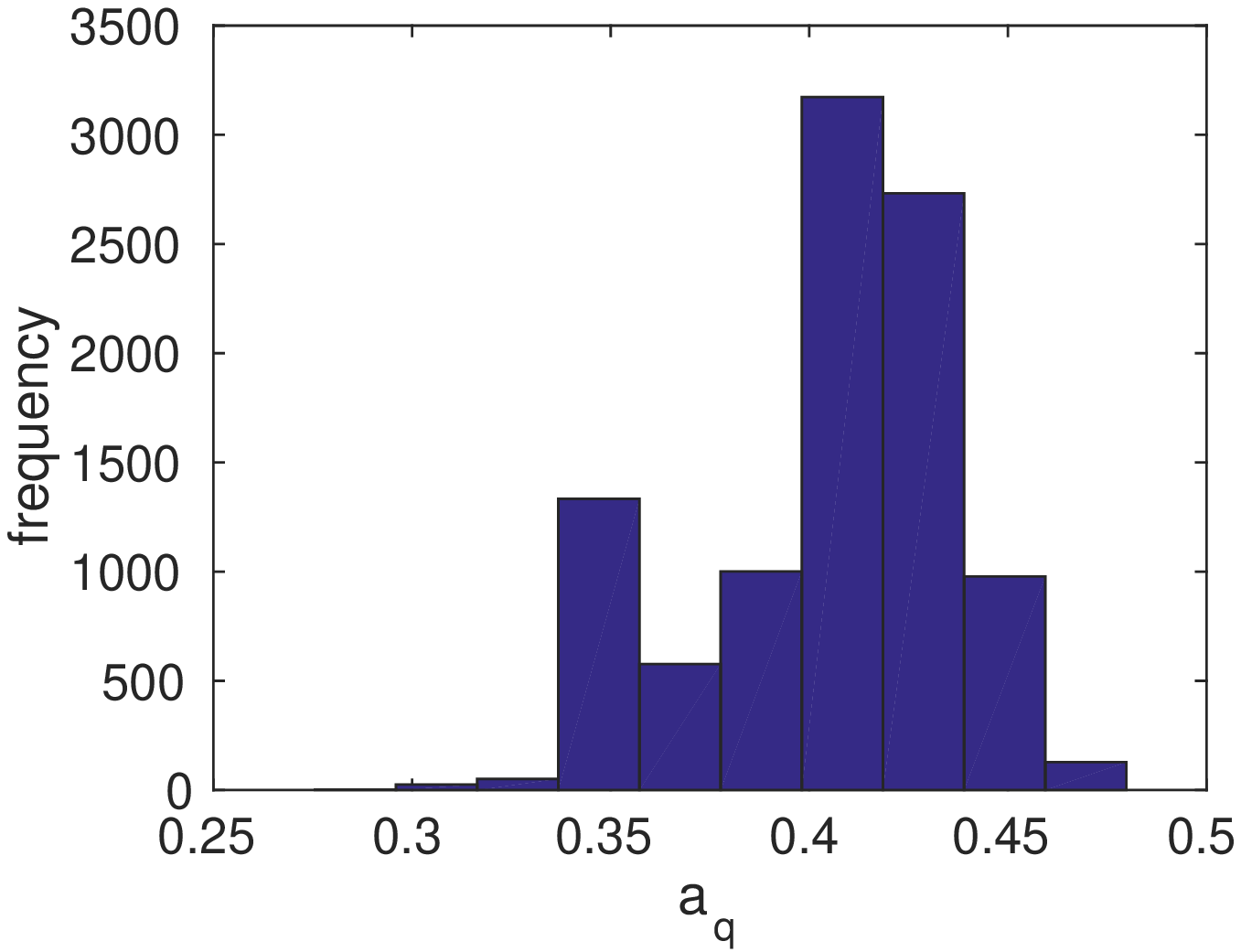,width=0.49\textwidth} 
\epsfig{file=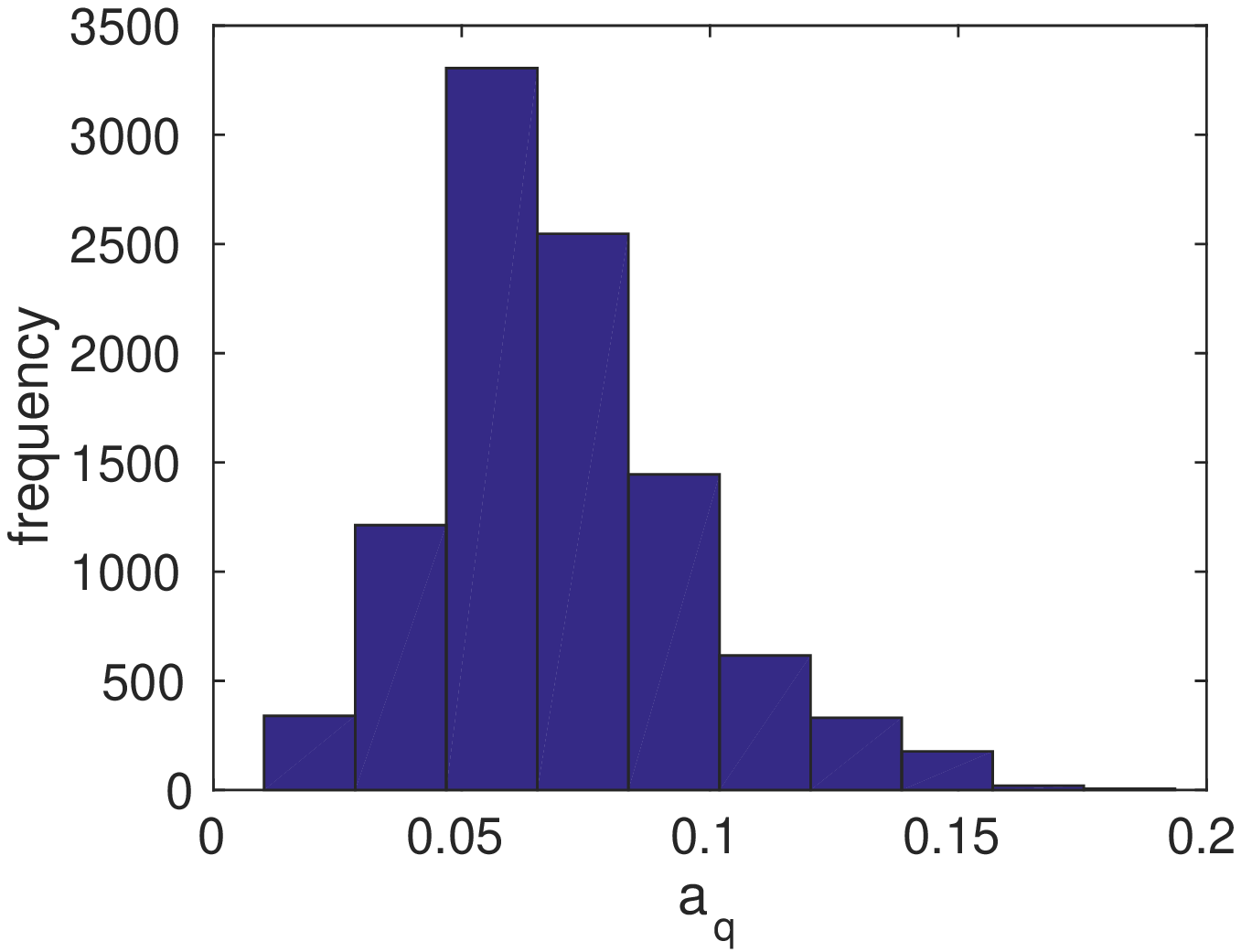,width=0.49\textwidth} 
\caption{(Top panels) The red asterisks give the normalized predator abundance $z(t)$ for simulations of smooth model I \eqref{smooth_1pred2prey} for parameter values $e=0.25$ and $\beta_1=\beta_2=1$ and fitted values of (left) $r_1\approx1.12$, $r_2\approx0.76$, $m\approx0.25$, $a_q\approx0.42$, and $k\approx38$ and (right) $r_1\approx2.32$, $r_2\approx0.51$, $m\approx0.27$, $a_q\approx0.027$, and $k\approx16$. To guide the eye, we show the simulation in red between the asterisks. We show the normalized data using blue circles, and we show blue lines between them to guide the eye. We also show bar plots for the frequency of (center panels) $k$ values and (bottom panels) $a_q$ values at the strictest tolerance level ($Tol_{15}\approx0.0213$ in the left panels and $Tol_{15}\approx0.0229$ in the right panels) using the PMC ABC method \cite{beaumont2009PMCABC} for (left) selective and (right) unselective predator groups in spring in Lake Constance in 1998. Each frequency plot represents a random weighted sample (of size 10000) from the PMC ABC's posterior distribution of the parameter values accepted at the strictest tolerance level. The squared distance [see Equation \eqref{squaredDistance}] between the asterisks (model prediction) and circles (data) is (left) 0.0090 and (right) 0.0061. For more details on  parameter fitting, see Appendix \ref{appendixParamFitting}. The unselective predator group consists of data for {\em Rimostrombidum lacustris}, and the selective predator group consists of data for {\em Balanion planctonicum}.} 
\label{simulationAndDataSmoothSys1_1998} 
\end{figure} 

%%%%%

\subsection{Comparison of smooth model II simulations and Lake Constance data}\label{subsection_comparisonwDataSysII}

To compare simulations of the smooth model \eqref{smoothSys2FullSystem} to data, we fit the prey growth rates $r_1$ and $r_2$, the predator mortality rate $m$, and a perturbation $\nu$ in the predator population from the coexistence steady state in Equation \eqref{coexistStStSSII} for $a_q=q_2=0.5$ (so that all four eigenvalues of the coexistence steady state are purely imaginary). We thus use $(p_1(0),p_2(0),z(0),q(0))=(a_qm(r_1+r_2)/[e(r_1a_q+r_2q_2)],m(r_1+r_2)/[e(r_1a_q+r_2q_2)],\nu(r_1+r_2),r_1/(r_1+r_2))$ as our initial value for the model simulations to infer values for $\nu$ that minimize the distance in Equation \eqref{squaredDistance} between the data points and the model prediction for these points. Thus, a small perturbation $\nu$ suggests a gradual diet change and $q$ oscillating around its equilibrium value, whereas one can interpret a large perturbation from the equilibrium as rapid changes in the diet (and the dynamics of $q$).

Smooth model II (just like smooth model I) reproduces the peak predator densities, and it seems that smooth model II best fits the data when there is a large perturbation from the coexistence steady state. See Figures \ref{simulationAndDataSmoothSys2_1991_normalisedData} and \ref{simulationAndDataSmoothSys2_1998_normalisedData}. For the year 1991, we predict that the selective predator group switches its diet less frequently than the unselective predator. Additionally, $q(t)$ reaches its maximum and minimum values. In contrast, for the unselective predator, there is a change from increasing $q(t)$ to decreasing in $q(t)$ at some intermediate value. See the dynamics of $q(t)$ at the bottom of the top panels of Figures \ref{simulationAndDataSmoothSys2_1991_normalisedData} and \ref{simulationAndDataSmoothSys2_1998_normalisedData}. We predict that the switching of the selective predator occurs more often in year 1998 than in year 1991 (and in 1998, it occurs also more often than that of the unselective predator). See Figure \ref{simulationAndDataSmoothSys2_1998_normalisedData}.

\begin{figure}[!] 
\centering
\epsfig{file=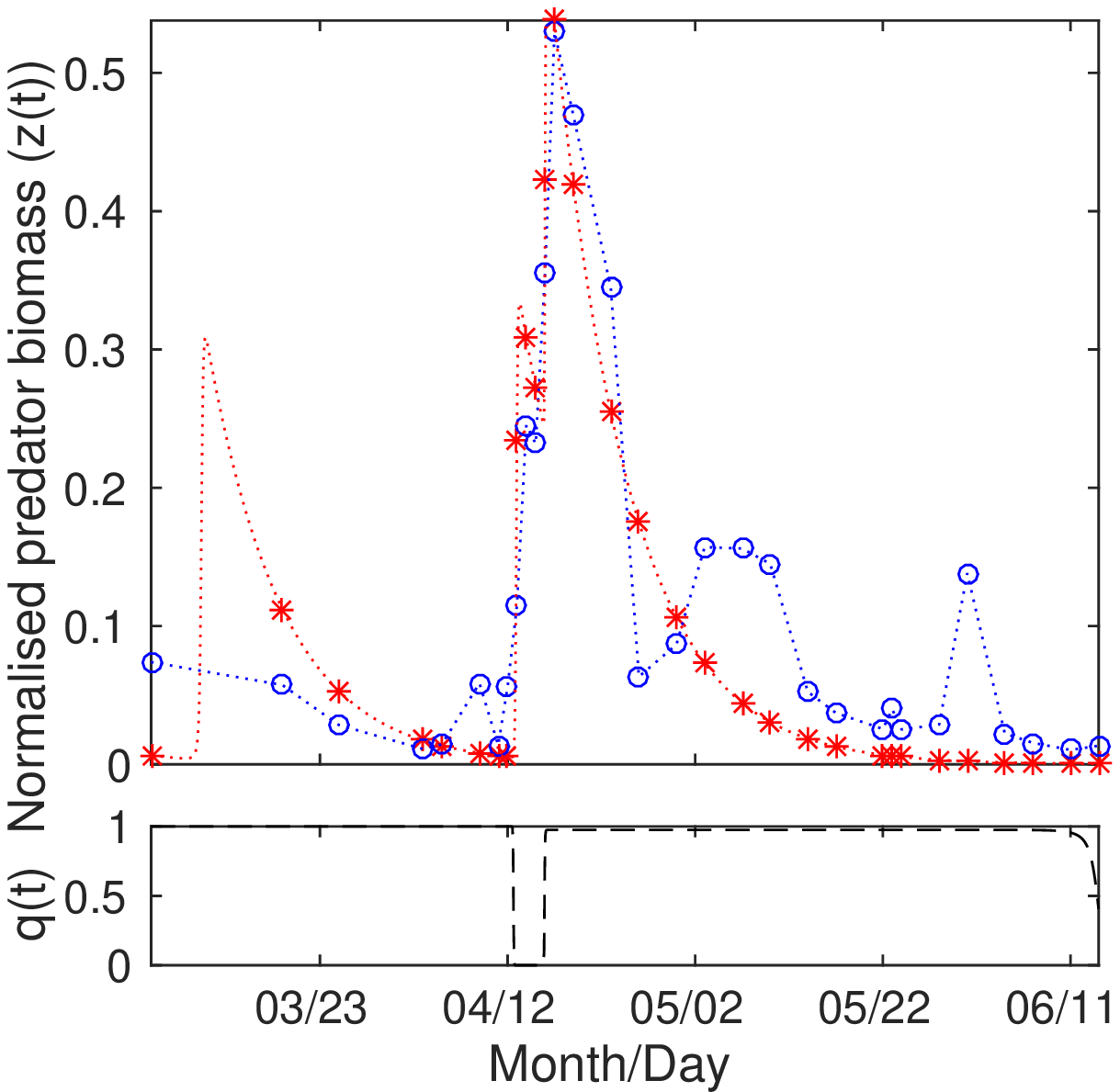,width=0.49\textwidth} 
\epsfig{file=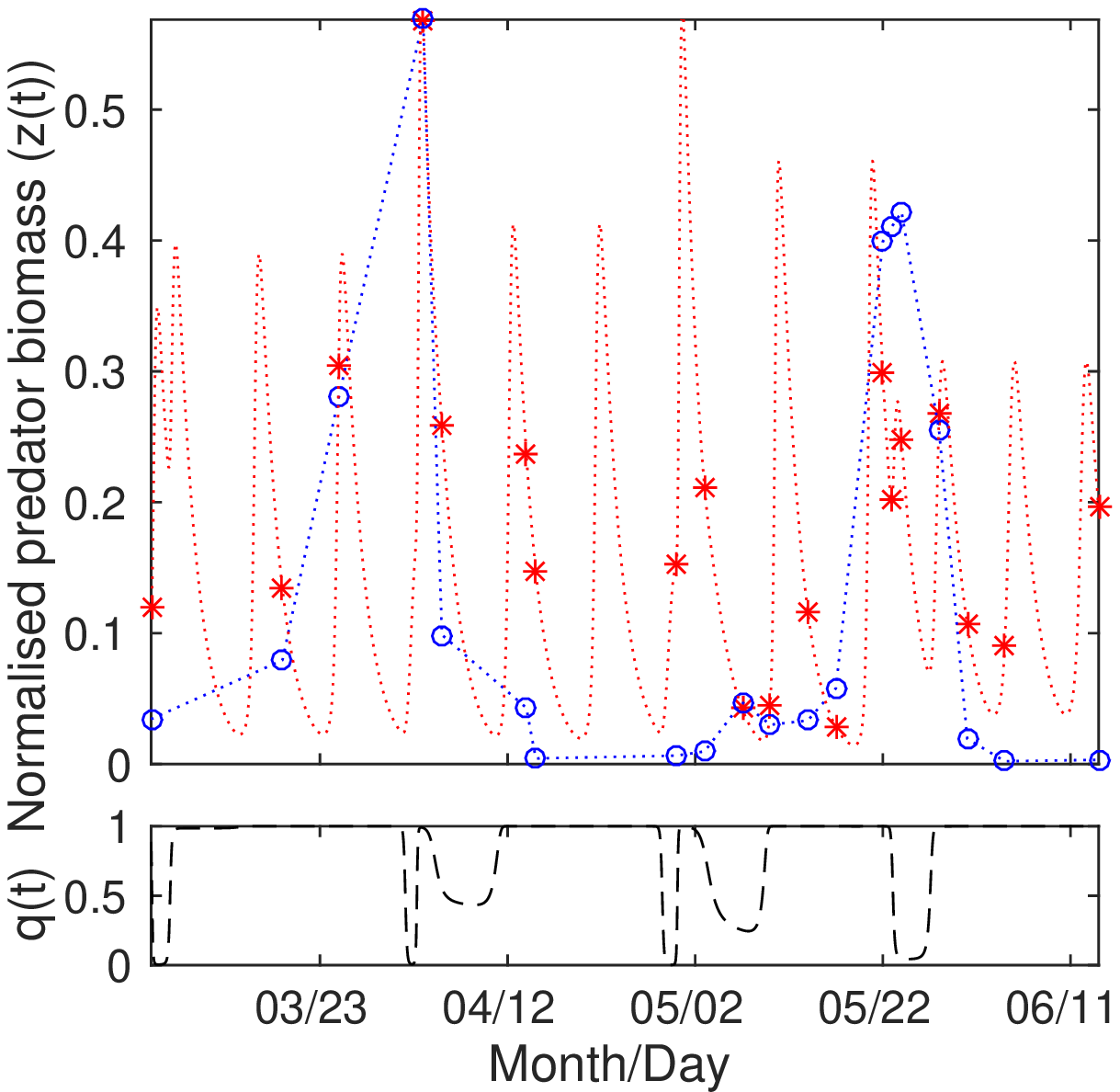,width=0.49\textwidth} 
\epsfig{file=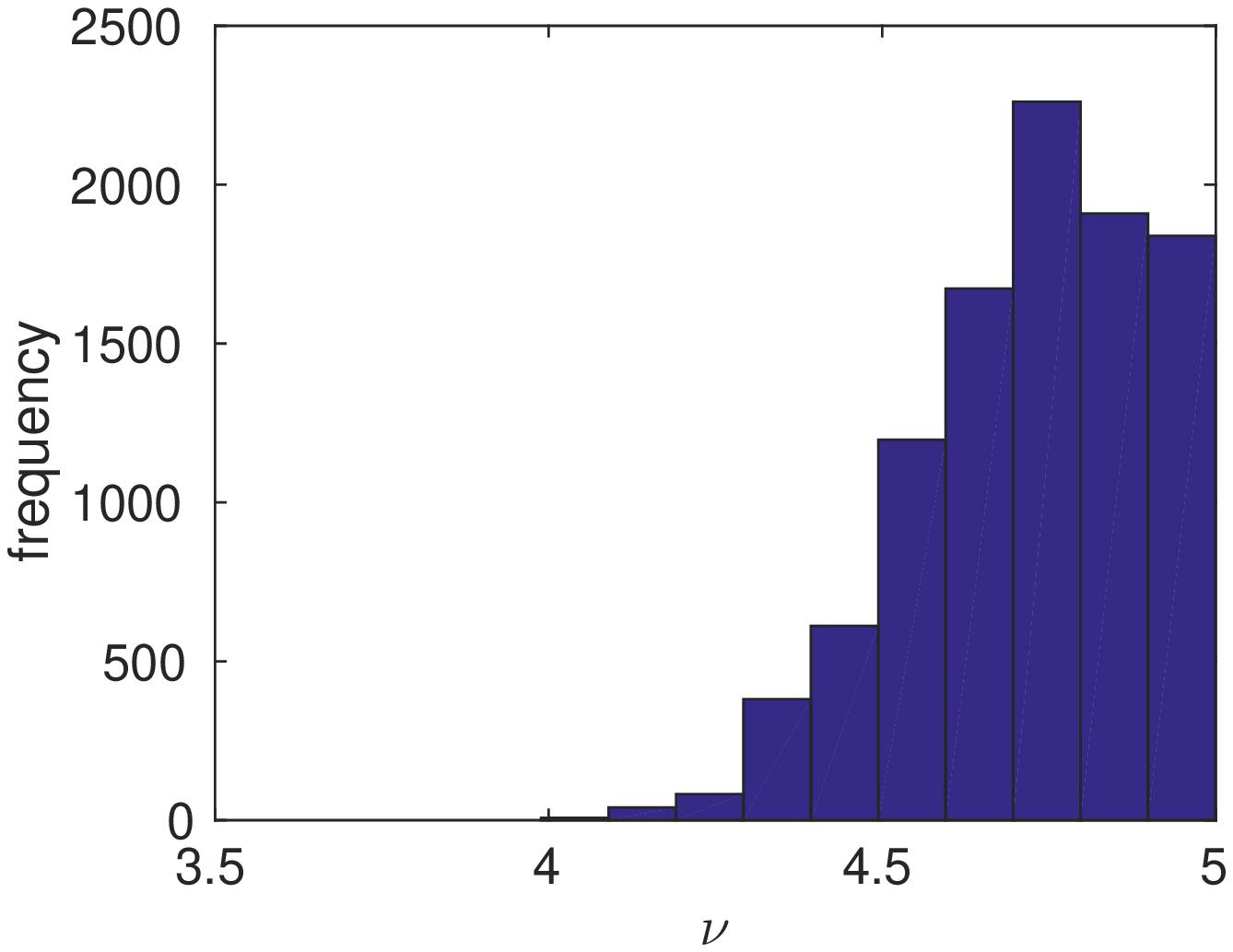,width=0.49\textwidth} 
\epsfig{file=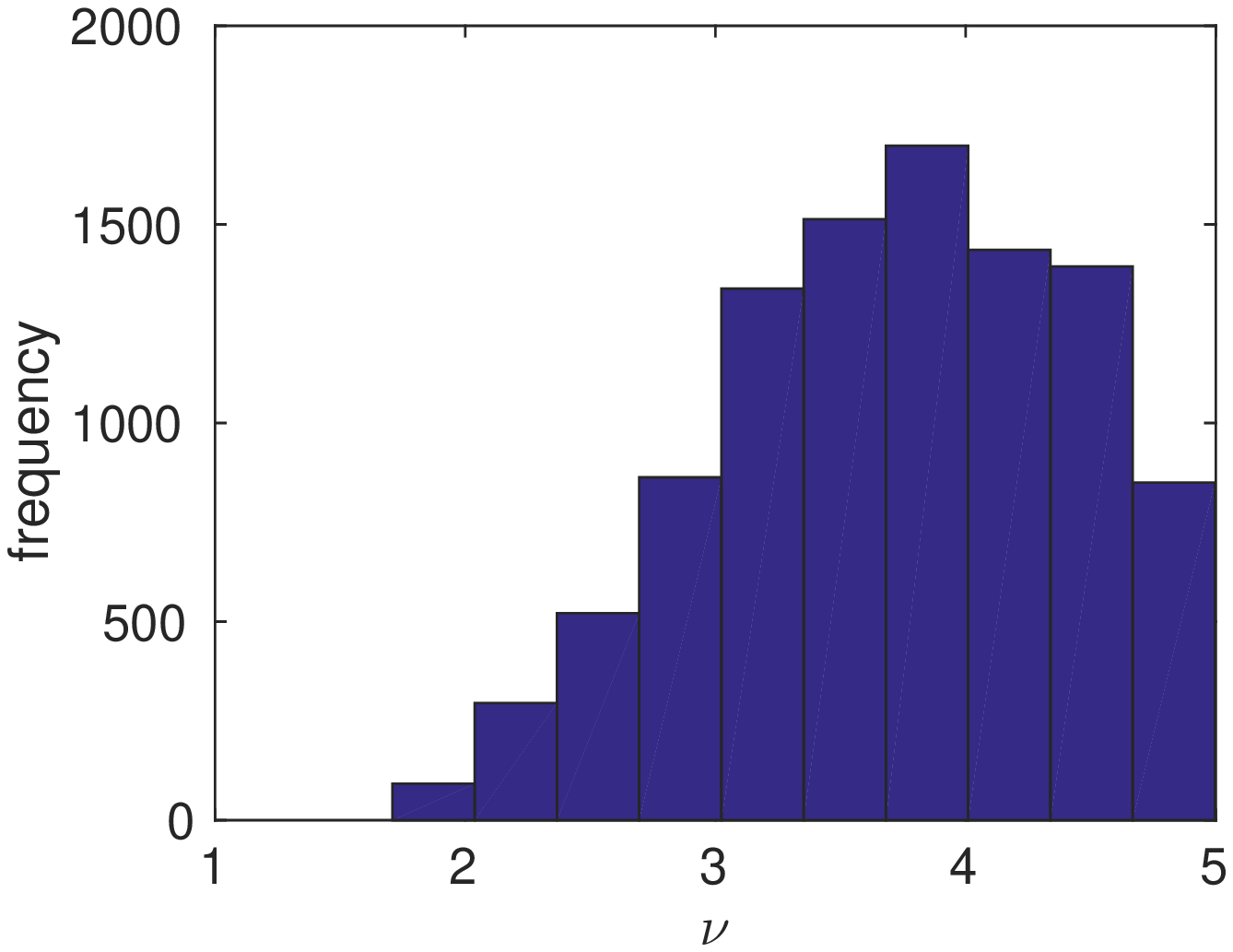,width=0.49\textwidth} 
\caption{(Top panels) The red asterisks give the normalized predator abundance $z(t)$ for simulations of smooth model II \eqref{smoothSys2FullSystem} with an initial value of $z(0)=\nu(r_1+r_2)$; the steady-state densities in \eqref{coexistStStSSII}; parameter values of $e=0.25$, $\beta_1=\beta_2=1$, and $a_q=q_2=0.5$; and fitted values of (left) $r_1\approx3.00$, $r_2\approx0.62$, $m\approx0.12$, and $\nu\approx 4.8$ and (right) $r_1\approx2.21$, $r_2\approx0.33$, $m\approx0.48$, and $\nu\approx4.6$. To guide the eye, we show the simulation in red between the asterisks. We show the normalized data using blue circles, and we plot blue lines between them to guide the eye. (Bottom panels) We show bar plots for the frequency of $\nu$ values at the strictest tolerance level ($\text{Tol}_{10}\approx0.00815$ in the left panel and $\text{Tol}_{13}\approx0.0233$ in the right panel) using the PMC ABC method \cite{beaumont2009PMCABC} for (left) selective and (right) unselective predator groups in spring in Lake Constance in 1991. Each frequency plot represents a random weighted sample (of size 10000)  from the PMC ABC's posterior distribution of the parameter values at the strictest tolerance level. The squared distance [see Equation \eqref{squaredDistance}] between the asterisks (model prediction) and circles (data) is (left) 0.0038 and (right) 0.0151. For more details on parameter fitting, see Appendix \ref{appendixParamFitting}. The unselective predator group consists of data for {\em Rimostrombidum lacustris}, and the selective predator group consists of data for {\em Balanion planctonicum}.} 
\label{simulationAndDataSmoothSys2_1991_normalisedData} 
\end{figure}

\begin{figure}[!] 
\centering
\epsfig{file=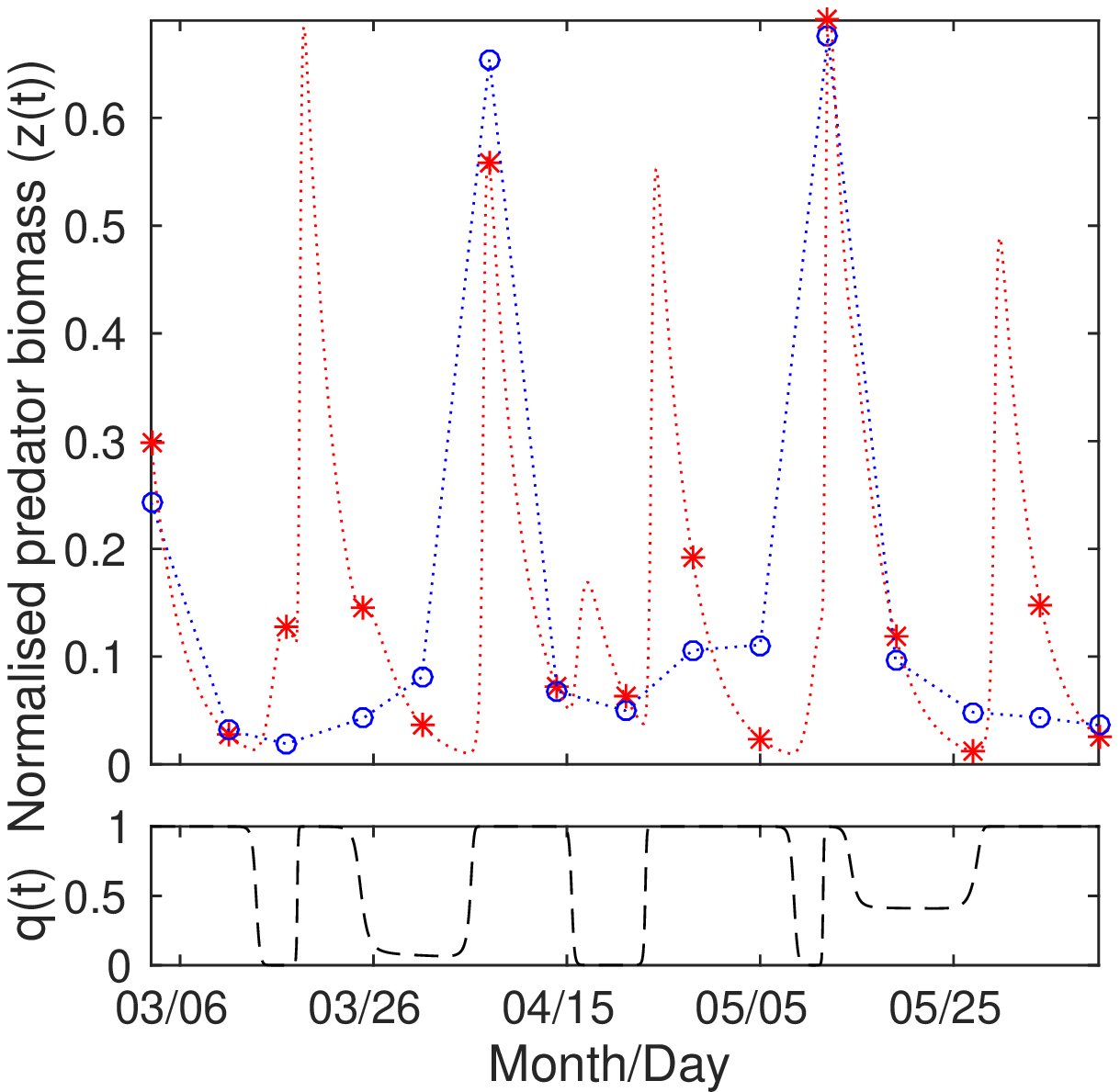,width=0.49\textwidth} 
\epsfig{file=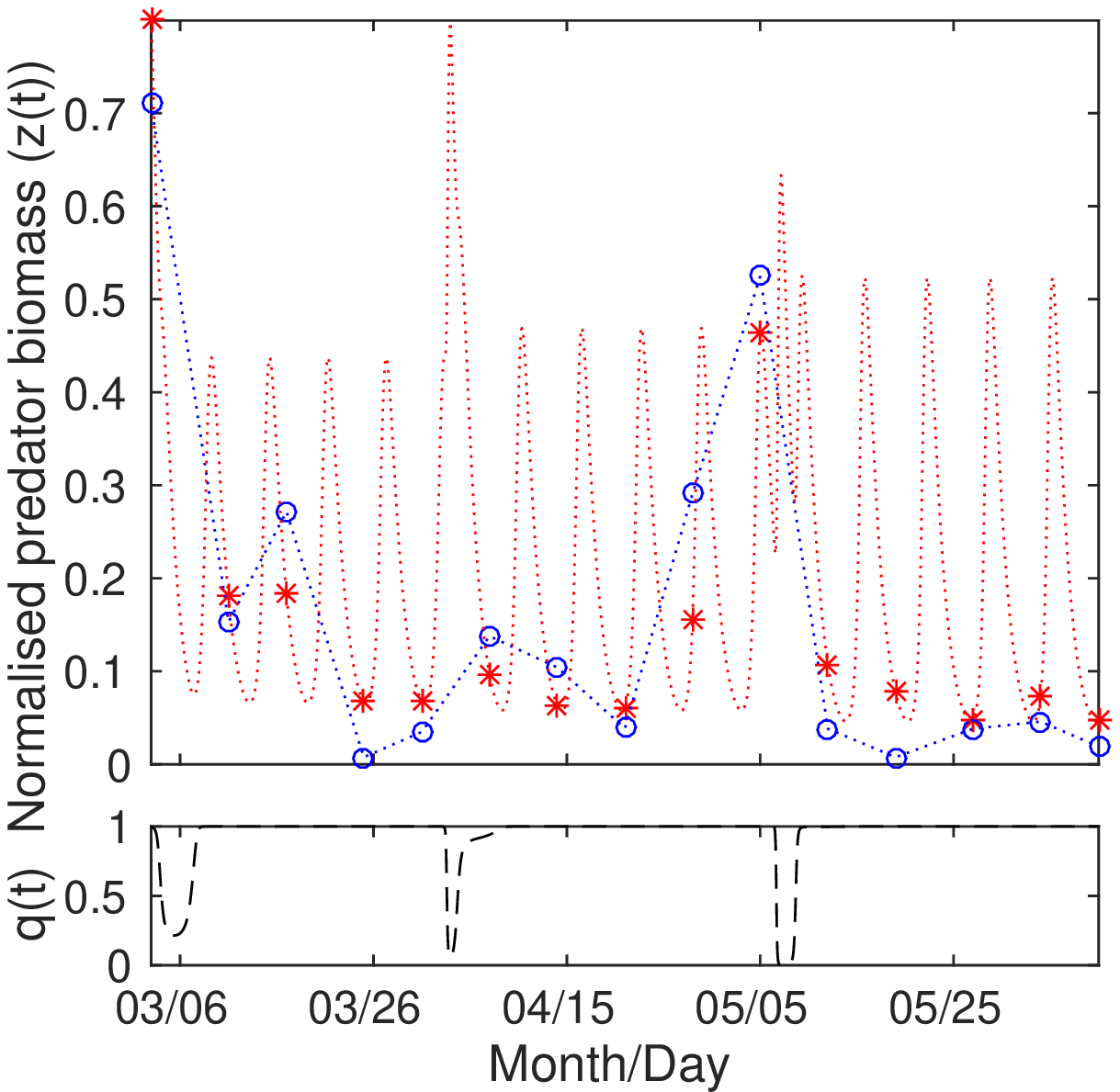,width=0.49\textwidth} 
\epsfig{file=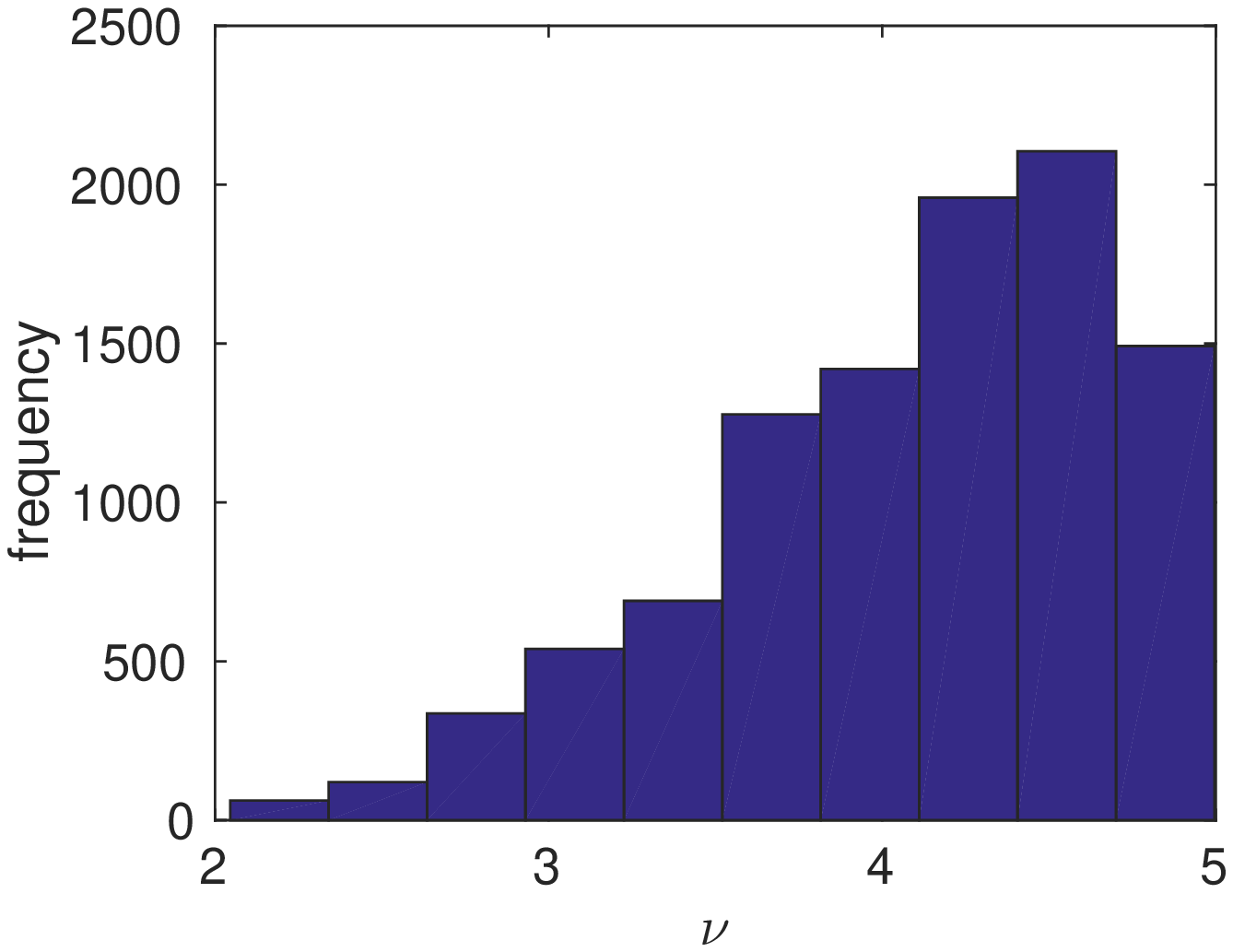,width=0.49\textwidth} 
\epsfig{file=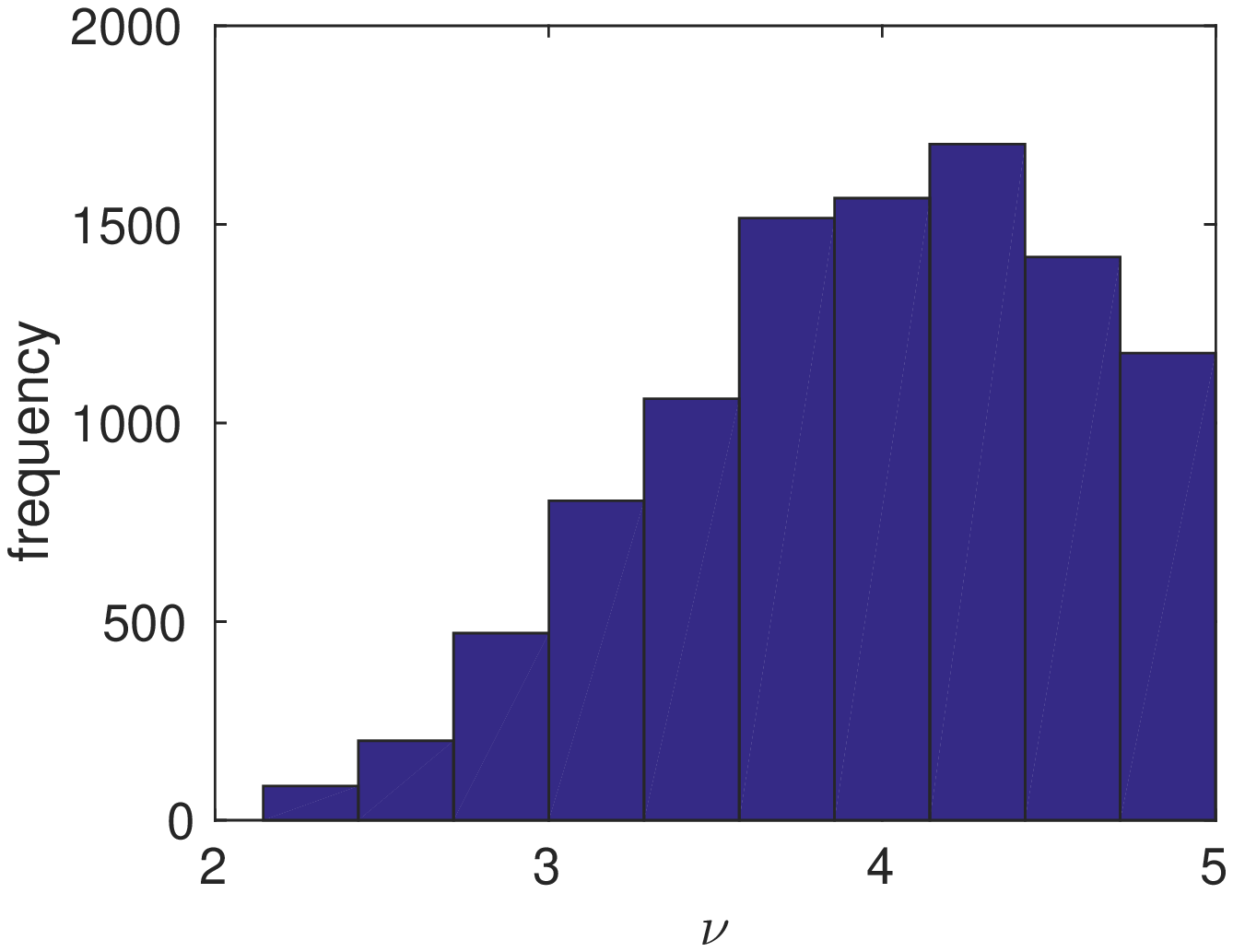,width=0.49\textwidth} 
\caption{(Top panels) The red asterisks give the normalized predator abundance $z(t)$ for simulations of smooth model II \eqref{smoothSys2FullSystem} with an initial value of $z(0)=\nu(r_1+r_2)$; the steady-state densities in \eqref{coexistStStSSII}; parameter values of $e=0.25$, $\beta_1=\beta_2=1$, and $a_q=q_2=0.5$; and fitted values of (left) $r_1\approx1.62$, $r_2\approx0.40$, $m\approx0.30$, and $\nu\approx4.8$ and (right) $r_1\approx1.95$, $r_2\approx0.24$, $m\approx0.72$, and $\nu\approx3.58$. To guide the eye, we show the simulation in red between the asterisks. We show the normalized data using blue circles, and we plot blue lines between them to guide the eye. (Bottom panels) We show bar plots for the frequency of $\nu$ values at the strictest tolerance level ($\text{Tol}_{13}\approx0.0245$ in the left panel and $\text{Tol}_{13}\approx0.0231$ in the right panel) using the PMC ABC method \cite{beaumont2009PMCABC} for (left) selective and (right) unselective predator groups in spring in Lake Constance in 1998. Each frequency plot represents a random weighted sample (of size 10000) from the PMC ABC's posterior distribution of the parameter values at the strictest tolerance level. The squared distance [see Equation \eqref{squaredDistance}] between the asterisks (model prediction) and circles (data) is (left) 0.0043 and (right) 0.0039. For more details on parameter fitting, see Appendix \ref{appendixParamFitting}. The unselective predator group consists of data for {\em Rimostrombidum lacustris}, and the selective predator group consists of data for {\em Balanion planctonicum}.}
\label{simulationAndDataSmoothSys2_1998_normalisedData} 
\end{figure}

To evaluate how well smooth model II \eqref{smoothSys2FullSystem} predicts prey abundance data (to which it was not fitted), we simulate it with parameter values that we obtain by fitting the model to the unselective predator in year 1991, and our model prediction for prey abundances suggests for an unselective predator in year 1991 that the preferred prey has smaller-amplitude oscillations than the alternative prey. As we show in Figure \ref{model2ValidationCheck_lessSel91}, this differs qualitatively from the data. We obtain the same result for smooth model I \eqref{smooth_1pred2prey} (comparison not shown). Nevertheless, although we only use predator data to fit parameters, smooth model II \eqref{smoothSys2FullSystem} is also able to successfully capture some features of the prey data. As we illustrate in the left panel of Figure \ref{model2ValidationCheck_lessSel91}, such features include the periodicity of the peak densities of the preferred prey populations.

\begin{figure}[!]
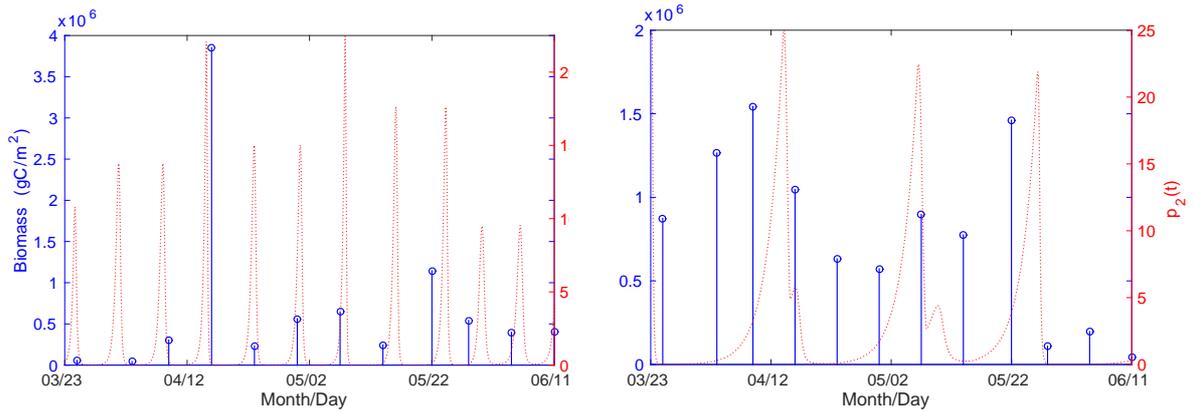
 
\centering
\epsfig{file=lessSel91_M2_p1_predAndData.eps,width=0.5\textwidth} 
\epsfig{file=lessSel91_M2_p2_predAndData.eps,width=0.47\textwidth} 
\caption{(Left, red) Preferred prey abundance $p_1(t)$ and (right, red) alternative prey abundance $p_2(t)$ for simulations of smooth model II \eqref{smoothSys2FullSystem} using the same parameter values as in the right panel of Figure \ref{simulationAndDataSmoothSys2_1991_normalisedData}. (Left, circles) preferred and (right, circles) alternative prey data in spring in Lake Constance in 1991. The preferred prey group consists of data for {\em Cryptomonas ovata}, {\em Cryptomonas marssonii}, {\em Cryptomonas reflexa}, {\em Cryptomonas erosa}, {\em Rhodomonas lens}, and {\em Rhodomonas minuta}. The alternative prey group consists of data for small and medium-sized {\em Chlamydomonas} spp. and {\em Stephanodiscus parvus}.
}
\label{model2ValidationCheck_lessSel91} 
\end{figure}

%%%%%

\section{Discussion}\label{section_discussion}

From a biological perspective, using a smooth dynamical system allows us to relax the assumption of a ``discontinuous'' predator of the piecewise-smooth system \eqref{1pred2prey}. When the discontinuity is smoothed out using hyperbolic tangent functions, as in smooth model I \eqref{smooth_1pred2prey}, we can use data to determine the steepness of the transition in the predator's feeding behavior for a particular predator type. Indeed, our parameter fitting to Lake Constance data suggests that one can model prey switching of either selective or unselective predator species with a steep hyperbolic tangent function. Additionally, our parameter fitting of smooth model I \eqref{smooth_1pred2prey} predicts that the best fit to the data occurs in the parameter regime in which the coexistence steady state is unstable. Additionally, simulations of our smooth model II \eqref{smoothSys2FullSystem}, which regularizes the abrupt change in the predator's diet choice by considering a predator trait as a system variable, exhibit rapid predator-trait dynamics (i.e., the temporal evolution of the predator's desire to consume the preferred prey $p_1$) suggesting that the best fit to data occurs when the change of diet is abrupt.  

From a modeling perspective, the piecewise-smooth system \eqref{1pred2prey} incorporates the effects of a predator's adaptive change of diet in response to prey abundance, whereas smooth system II \eqref{smoothSys2FullSystem} (with an appropriate choice of parameters) explores rapid evolutionary change in a predator's desire to consume its preferred prey \cite{ourfastslowPaper}. Consequently, smooth system II models a different mechanism (namely, \emph{rapid evolution} \cite{Fussman2007}) than the piecewise-smooth system (which models \emph{phenotypic plasticity} \cite{Kellyetal_phenplast2012}) for how rapid adaptation affects population dynamics \cite{Shimadaetal2010,Yamamichietal2011}. It has been suggested that one should be more likely to expect a stable equilibrium from models that account for phenotypic plasticity than in those that account for rapid evolution, because plastic genotypes respond faster than nonplastic genotypes to fluctuating environmental conditions \cite{Yamamichietal2011}. Our modeling work is consistent with this hypothesis, as the piecewise-smooth system \eqref{1pred2prey} converges to a steady state for a large region of phase space \cite{OurPaper}, but the same steady state is unstable --- except for one point ($a_q=q_2$), at which it is stable but nonhyperbolic --- in smooth model II \eqref{smoothSys2FullSystem}.

Smooth model I \eqref{smooth_1pred2prey} necessitates the incorporation of a parameter $k$ that influences the system's qualitative behavior, whereas smooth system II \eqref{smoothSys2FullSystem} has the same number of parameters as the piecewise-smooth system \eqref{1pred2prey} but includes an additional system variable. It can thus be advantageous to study the piecewise-smooth system, especially if many species are included, because it allows one to avoid adding new parameters and/or variables. The hyperbolic tangent functions in \eqref{smooth_1pred2prey} and the increased dimensionality of \eqref{smoothSys2FullSystem} both add complications to analytical calculations and parameter fitting. On the plus side, there are many more standard numerical techniques and standard theory to determine the stability of equilibria using linear stability analysis and study bifurcations for smooth dynamical systems than for piecewise-smooth ones. One needs to use more involved methods for theory and numerical computations for piecewise-smooth dynamical systems, and development of these techniques is an active area of research \cite{pw-sBook}. However, one can derive an analytical expression (specifically, $h=\beta_1p_1-a_q\beta_2p_2=0$) for the flow at the discontinuity boundary of \eqref{1pred2prey}, and the available theory for piecewise-smooth dynamical systems identifies the bifurcation that takes place in \eqref{1pred2prey} as $a_q$ crosses the value $a_q=q_2/q_1$ \cite{OurPaper}. These results help facilitate understanding of the behavior of \eqref{1pred2prey} and can be used when analyzing the ciliate--algae dynamics predicted by the piecewise-smooth model \cite{OurPaper}.

Both of our smooth models successfully reproduce the peak population densities and suggest that a parameter regime --- when $a_q<q_2/q_1$ for smooth model I \eqref{smooth_1pred2prey} and for a large perturbation from the steady state for smooth model II \eqref{smoothSys2FullSystem} --- fits the data for ciliate predators in Lake Constance in the springs of 1991 and 1998. (Note that our initial distributions for these parameters in the fitting algorithm include both parameter regimes in which the coexistence steady state is stable and ones in which it is unstable.) Additionally, when using the parameters that we obtain from fitting smooth model II to data for the unselective predator in year 1991, we obtain agreement between our model's prediction and the periodicity of the peak prey abundances. Both of our smooth models predict a higher frequency of peak densities than what we observe in the available data points. A high period in population oscillations is possible for small organisms, such as plankton, with short lifespans and large population densities. Making measurements more frequently than biweekly would be a good way to try to validate 
 or refute the periodicity predicted by the smooth models. Additionally, using data comparison to choose between different competing models would be an effective way to help increase current understanding of the use of a piecewise-smooth model as a simplification of a steep transition in plankton-feeding behavior. Such comparisons would also be valuable more generally in numerous applications. In practice, one can carry out such a model comparison by implementing algorithms of model-based statistical inference methods (e.g., approximate Bayesian computation, as in the present paper) \cite{Toni2009} or by using existing toolboxes for system identification (e.g., the ones implemented in {\sc{Matlab}} \cite{MATLAB2014}). 

 One can further investigate model predictions for the trait dynamics and compare them to results from controlled laboratory experiments with genetically diverse prey and/or predator populations in which one records the dynamics of the genetic diversity. Parameter fitting to Lake Constance data suggests that the best fit occurs in a parameter regime in which the predator trait dynamics oscillate abruptly between the maximum and minimum values. In a study of two plankton predators and their evolving algal prey, Hiltunen et al. \cite{Hiltunenetal} showed and discussed experimental evidence for periods of dominance of one predator followed by a rapid switch to a dominance by the other. In \cite{Hiltunenetal}, the switch in predator dominance arose from interactions between changes in the predator population and changes in the frequency of the prey type that develops a predator defense mechanism against one of the two predators. Motivated by the above findings, it is also interesting to consider a model in which one incorporates a time-scale difference between demographic and predator-trait dynamics \cite{ourfastslowPaper}. 

%%%%%

\section{Conclusions}\label{section_conclusions}

To increase biological insight into the experimentally-observed adaptive feeding behavior of unselective and selective ciliate predators on two different types of prey, we constructed two ordinary differential-equation models for prey switching. In one model (``smooth model I''), we represent the transition from one diet to another using the hyperbolic tangent function; in the other (``smooth model II''), we added a new system variable to describe the diet switch in the system (and we hence increased the system's dimension by $1$). In constructing these models, we relaxed the simplifying assumption of a ``discontinuous'' predator feeding behavior in a piecewise-smooth dynamical system that was used previously to suggest prey switching as a putative mechanistic explanation for the observed dynamics \cite{OurPaper}. Based on our results from fitting parameters of the two smooth systems to data on freshwater plankton, we conclude that the best fit to the data occurs when prey switching is rapid (and hence steep in the continuous models) and that the simplifying assumption of discontinuous predator feeding behavior appears to be justified. 

Similar to earlier investigations, such as \cite{jeffrey2011nondeterminism,leifeld2015persistence}, our study provides an illustrative example of both similarities and differences between a discontinuous system and smooth regularizations to it. When piecewise-smooth dynamical systems are used to simplify transitions in applications --- such as approximating a cubic function in a membrane potential in models of spiking neurons \cite{mckean1970nagumo}, Hill functions in models of gene regulatory networks \cite{Glass1975}, changes in the Earth's reflectivity due to ice melt in climate models \cite{abbot2011jormungand}, and more --- understanding the extent to which the behavior of the corresponding smooth and piecewise-smooth systems agree is crucial for generating both accurate model simplifications and accurate predictions.

Finally, using the data that we currently possess, it is difficult to determine which of the three models (i.e., a piecewise-smooth model or the two smooth systems with an adaptive predator) provides a better mechanistic explanation for the observations of ciliate--algae dynamics in spring in Lake Constance. To enhance model selection, it would be very useful to collect data to improve analysis of the steepness of prey switching, the functional form of the preference tradeoff, and the periodicity of the population oscillations. Nevertheless, the construction of models using alternative modeling approaches, transforming between them, and comparing them to data can greatly increase understanding of the underlying mechanisms of biological systems. 

%%%%

\section*{Acknowledgements} 

We thank John Hogan, Philip Maybank, Frank Schilder, and Frits Veerman for helpful discussions. We thank Ursula Gaedke for sending us the Lake Constance data, which were obtained as part of the Collaborative Programme SFB 248 funded by the German Science Foundation. SHP was supported by Osk. Huttunen Foundation (OHF) and Engineering and Physical Sciences Research Council (EPSRC) through the Oxford Life Sciences Interface Doctoral Training Centre and by People Programme (Marie Curie Actions) of the European Union's Seventh Framework Programme (FP7/2007-2013) under REA grant agreement \#609405 (COFUNDPostdocDTU). PKM would like to thank the Mathematical Biosciences Institute (MBI) at Ohio State University, for partially supporting this research. MBI receives its funding through the National Science Foundation grant DMS1440386.

%%%%

\appendix{}

\section{Stability of the coexistence steady state in smooth model I}\label{appendixStabilityI}

In this appendix, we prove Propositions \ref{propositionSmoothModelI_1} and \ref{propositionSmoothModelI_2}.
\begin{proof}
The Jacobian for Equation \eqref{smooth_1pred2prey} is 
\begin{equation}
	\begin{bmatrix}
r_1-\frac{z}{2}B\left(1+p_1kC\right) & \frac{zp_1ka_q}{2}BC & \frac{-p_1}{2}B\\
\frac{zp_2k}{2}BC & r_2-\frac{z}{2}C\left(1+p_2ka_qB\right) & \frac{-p_2}{2}C\\
\frac{eq_1z}{2}B\left(1+p_1kC\right)-\frac{eq_2p_2zk}{2}BC & \frac{eq_2z}{2}C\left(1+p_2ka_qB\right)-\frac{eq_1p_1zka_q}{2}BC & \frac{eq_1p_1}{2}B+\frac{eq_2p_2}{2}C-m\\
	\end{bmatrix} \label{JacobianoftheSmoothSys1}\,,
\end{equation}
where 
\begin{align}
	A&=\text{tanh}\left(k\left(p_1-a_qp_2\right)\right)\,, \nonumber \\ 
	B&=1+A\,, \\
	C&=1-A\,.\nonumber
\end{align}
At the coexistence steady state \eqref{coexistenceStStSmoothSysI}, the Jacobian \eqref{JacobianoftheSmoothSys1} is given by
\begin{equation}
	\frac{1}{r_1+r_2}\begin{bmatrix}
-2r_1r_2k\tilde{p}_1 & 2r_1r_2ka_q\tilde{p}_1 & -r_1\tilde{p}_1\\
2r_1r_2k\tilde{p}_2 & -2r_1r_2ka_q\tilde{p}_2 & -r_2\tilde{p}_2\\
eq_1r_1(r_1+r_2)+2er_1r_2k(q_1\tilde{p}_1-q_2\tilde{p}_2) & eq_2r_2(r_1+r_2)+2er_1r_2ka_q(q_2\tilde{p}_2-q_1\tilde{p}_1) & 0\\
	\end{bmatrix} \label{JacobianoftheSmoothSys1_reducedOnce}\,,
\end{equation}
which we henceforth denote by $J$ for the rest of the present appendix. The characteristic polynomial of \eqref{JacobianoftheSmoothSys1_reducedOnce} is
\begin{equation}
	 |\lambda I - J| = \lambda^3 + a \lambda^2 + b \lambda + c\,,
\end{equation}
where 
\begin{equation}
	\begin{split}
		a &= \frac{2 k (p_1 + a_q p_2) r_1 r_2}{r_1 + r_2}\,, \\
		b &= \frac{1}{(r_1 + r_2)^2} e \left(2 k p_1^2 q_1 r_1^2 r_2 + 
   p_2 q_2 r_2^2 (r_1 + 2 a_q k p_2 r_1 + r_2) \right. \\
   &\quad \left. + p_1 r_1 (-2 k p_2 q_2 r_1 r_2 + q_1 (r_1^2 + r_1 r_2 - 2 a_q k p_2 r_2^2))\right)\,, \\
	   c &= \frac{2 e k p_1 p_2 r_1 r_2 (a_q q_1 r_1 + q_2 r_2)}{r_1 + r_2}\,.
	\label{eqn:coefficients}
\end{split}
\end{equation}
According to the Routh--Hurwitz criterion \cite{routh1877,hurwitz1895}, the coexistence steady state \eqref{coexistenceStStSmoothSysI} is stable if and only if the coefficients in \eqref{eqn:coefficients} satisfy $a>0$, $c>0$, and $a b-c>0$. The conditions $a>0$ and $c>0$ are satisfied because of the positivity of the system parameters. To study the third condition, we write $ab-c$ as a polynomial in $k$. We thereby obtain
\begin{align*}
	a b-c &= 
 \frac{2 r_1 r_2}{e (a_q q_1 r_1 + q_2 r_2)^2 (r_1 + r_2)^2}
 \left( p_1 r_1 - a_q p_2 r_2 \right)
s(k) \,,
\end{align*}
where
\begin{equation}
	s(k) = s_2 k^2 + s_1 k + s_0\,, 
\end{equation}
with
\begin{align*}
	s_0 &= \frac12 e^2 q_1 q_2 r_1 r_2 \log\left(\frac{r_1}{r_2}\right)\left[2 a_q q_1 r_1 + 
	2 q_2 r_2 + (-a_q q_1 r_1 + q_2 r_2) \log\left(\frac{r_1}{r_2}\right)\right]\,, \\
     s_1 &= e m \left[(r_1 + r_2) (a_q^2 q_1^2 r_1^2 - q_2^2 r_2^2) - 
     r_1 r_2 (a_q^2 q_1^2 r_1 + q_2^2 r_2 - 3 a_q q_1 q_2 (r_1 + r_2)) \log\left(\frac{r_1}{r_2}\right)\right]\,,\\
     s_2 &= 4 a_q m^2 (a_q q_1 - q_2) r_1 r_2 (r_1 + r_2)\,.
\end{align*}
Using the steady state \eqref{coexistenceStStSmoothSysI}, we see that
\begin{equation*}
		p_1 r_1 - a_q p_2 r_2 = \frac{a_q k m (r_1^2 - r_2^2) + 
			e (a_q q_1 + q_2) r_1 r_2 \text{arctanh}\left(\frac{r_1 - r_2}{r_1 + r_2}\right)}{e k (a_q q_1 r_1 + 
			q_2 r_2)} > 0\,.
\end{equation*}
We have thus established that $a b -c > 0 \Leftrightarrow s(k) > 0$.

The value of $s$ at $k = k_0$ in Equation \eqref{equationFork0} is positive:
\begin{equation}
	s(k_0) = \frac{e^2 q_1 r_1 (a_q q_1 r_1 + q_2 r_2)^2 \log\left(\frac{r_1}{r_2}\right)\left[r_1 + r_2 + r_2 \log\left(\frac{r_1}{r_2}\right)\right]}{2 (r_1 + r_2)} > 0 \,.
\end{equation}
For $a_q \geq q_2/q_1$, the function $s(k)$ is concave up with a positive derivative at $k_0$, because
\begin{equation}
	s'(k_0) = e m (a_q q_1 r_1 + 
	q_2 r_2)\left[(r_1 + r_2) (a_q q_1 r_1 - q_2 r_2) + (3 a_q q_1 - q_2) r_1 r_2 \log\left(\frac{r_1}{r_2}\right)\right] > 0\,.
\end{equation}
Therefore, $a_q \geq q_2/q_1$ implies that $s(k) > 0$ for all $k > k_0$.
Finally, if $a_q < q_2/q_1$, we see that $s(k)$ is a downward-opening parabola. Because $s(k_0)>0$, it follows that $s$ is positive in the interval $(k_0,k_1)$, where
\begin{equation}
	k_1 = \frac{-s_1 + \sqrt{s_1^2 - 4 s_2 s_0}}{2 s_2}\,. 
	\label{expressionOfK1}
\end{equation}
\end{proof}

%%%%%

\section{Stability of the coexistence steady state in smooth model II}\label{appendixStabilityII}

In this appendix, we prove Propositions \ref{propositionSmoothModelII_1} and \ref{propositionSmoothModelII_2}. Both proofs use the characteristic polynomial of the Jacobian of smooth system II \eqref{smoothSys2FullSystem}. This Jacobian is
\begin{align}
	J &=
\begin{pmatrix}
r_1-\tilde{q}\tilde{z} & 0 & -\tilde{q}\tilde{p}_1 & -\tilde{p}_1\tilde{z} \\
0 & r_2-(1-\tilde{q})\tilde{z} & -(1-\tilde{q})\tilde{p}_2 & \tilde{p}_2\tilde{z} \\
e\tilde{q}\tilde{z} & e(1-\tilde{q})q_2\tilde{z} & e\tilde{q}\tilde{p}_1+e(1-\tilde{q})q_2\tilde{p}_2-m & e\tilde{p}_1\tilde{z}-eq_2\tilde{p}_2\tilde{z} \\
\tilde{q}(1-\tilde{q}) & -a_q\tilde{q}(1-\tilde{q}) & 0 & (1-2\tilde{q})(\tilde{p}_1-a_q\tilde{p}_2)
\end{pmatrix}\,.
	\label{smoothSysIIJacobian}
\end{align}
At the coexistence steady state \eqref{coexistStStSSII}, the Jacobian \eqref{smoothSysIIJacobian} is given by
\begin{equation}
J=
\begin{pmatrix}
	0 & 0 & \frac{-r_1a_qm}{e(r_1a_q+r_2q_2)} & \frac{-a_qm(r_1+r_2)^2}{e(r_1a_q+r_2q_2)} \\
	0 & 0 & \frac{-r_2m}{e(r_1a_q+r_2q_2)} & \frac{m(r_1+r_2)^2}{e(r_1a_q+r_2q_2)} \\
	er_1 & er_2q_2 & 0 & \frac{e(a_q-q_2)m(r_1+r_2)^2}{e(r_1a_q+r_2q_2)} \\
\frac{r_1r_2}{(r_1+r_2)^2} & \frac{-a_qr_1r_2}{(r_1+r_2)^2} & 0 & 0
\end{pmatrix} \,,
	\label{smoothSysIIJacobianreduced}
\end{equation}
whose eigenvalues are given by the roots of the characteristic polynomial
\begin{align}
  \text{det}(\lambda I-J)=\lambda^4&+\frac{m\left(eq_2r_2^2+a_qr_1(er_1+2r_2)\right)}{e(a_qr_1+q_2r_2)}\lambda^2 \label{smoothSys2_VststLambdasCharEqn_aqneqq2}\\
  &+\frac{a_qr_1r_2m^2(a_q-q_2)(r_1-r_2)}{e(a_qr_1+q_2r_2)^2}\lambda+\frac{a_qr_1r_2m^2(r_1+r_2)}{e(a_qr_1+q_2r_2)}\,.
\nonumber
\end{align}

\vspace{.5 cm}

\begin{proof} (Proposition \ref{propositionSmoothModelII_1})

For $a_q=q_2$, the $\mathcal{O}(\lambda)$ term in Equation \eqref{smoothSys2_VststLambdasCharEqn_aqneqq2} vanishes. Substituting $u=\lambda^2$ yields
\begin{equation}
 	 u^2 +\frac{m[2r_1r_2+e(r_1^2+r_2^2)]}{e(r_1+r_2)}u+\frac{m^2r_1r_2}{e}\,.
 \label{uSecondOrderEqnForLambda}
\end{equation}
One can write the discriminant of (\ref{uSecondOrderEqnForLambda}) as  
\begin{equation}
	 D=\frac{m^2}{e^2(r_1+r_2)^2} \left((r_1^2+r_2^2)^2\left(e-\frac{4r_1^2r_2^2}{(r_1^2+r_2^2)^2}\right)^2+\frac{4r_1^2r_2^2(r_1^2-r_2^2)^2}{(r_1^2+r_2^2)^2}\right)\,.
\end{equation}
Note that $D$ is always positive, so the two roots ($u_1$ and $u_2$) of \eqref{uSecondOrderEqnForLambda} are both real. Furthermore, because the polynomial (\ref{uSecondOrderEqnForLambda}) is increasing and positive at the intersection of (\ref{uSecondOrderEqnForLambda}) with the vertical axis, the roots of the polynomial (\ref{uSecondOrderEqnForLambda}) are both negative. Consequently, the four eigenvalues $\lambda_j$ (with $j \in \{1,2,3,4\}$) consist of two complex-conjugate pairs with $0$ real part: $\lambda_{1,2}=\pm\sqrt{u_1}i$ and $\lambda_{3,4}=\pm\sqrt{u_2}i$. We thus see that all eigenvalues are purely imaginary.

\end{proof}

\vspace{.5 cm}

\begin{proof} (Proposition \ref{propositionSmoothModelII_2})

First, we prove by contradiction that there is at least one eigenvalue with a non-zero real part. Assume that all four eigenvalues are purely imaginary. One can then write the characteristic polynomial \eqref{smoothSys2_VststLambdasCharEqn_aqneqq2} as
\begin{equation}
	\chi(\lambda) = \prod_{j=1}^4 (\lambda - \imath y_j)\,, \quad y_j \in \mathbb{R}\,. 
\end{equation}
Expanding $\chi$, we see that the $\mathcal{O}(\lambda)$ coefficient is 
\begin{equation*}
 	\imath (y_1 y_2 y_3 + y_1 y_2 y_4 + y_1 y_3 y_4 + y_2 y_3 y_4)\,,
\end{equation*}	
which is purely imaginary. However, \eqref{smoothSys2_VststLambdasCharEqn_aqneqq2} has a real coefficient for $\mathcal{O}(\lambda)$ that is non-zero for $a_q \neq q_2$. Therefore, there exists at least one eigenvalue with a non-zero real part.

To complete the proof, we show that there are two eigenvalues whose real parts have opposite signs. We denote the roots of  \eqref{smoothSys2_VststLambdasCharEqn_aqneqq2} by $\lambda_j$ (with $j \in \{1,2,3,4\}$), and we order the roots so that the real part of $\lambda_1$ is non-zero. Because $\sum_{j=1}^4 \lambda_j = \text{Tr }(J) = 0$, at least one of $\lambda_2$, $\lambda_3$, or $\lambda_4$ must have a real part whose sign is opposite to that of $\lambda_1$. Consequently, the steady state is unstable. 
\end{proof}

%%%%%
\section{Parameter fitting}\label{appendixParamFitting}

We study the parameter-fitting problem through Bayesian inference. In contrast to least-squares fitting, this allows one to study the results from the posterior parameter distribution rather than just a single value that gives the best fit as a result of an optimization method. Because our modeling results in a posterior associated to the data, we
fit parameters to data with approximate Bayesian computation (ABC) combined with a population Monte Carlo (PMC) method. (See p.~987 of \cite{beaumont2009PMCABC}.)

%%%%%

\subsection{Smooth model I} \label{section:paramFittingM1}

Let $z(t;r_1,r_2,m,a_q,k)$ denote the solution of \eqref{smooth_1pred2prey} with initial values $(p_1(0),p_2(0),z(0)) = (1,1,1)$, and let $\mathbf z$ denote the available measurement data on the predator population. The data were measured at time instances $t_i$, so --- without measurement errors --- the data would be ${\mathbf z}_i = z(t_i;r_1,r_2,m,a_q,k)$ for some unknown, true parameter values. We account for the presence of measurement errors by incorporating normally-distributed noise into the results of our model simulations. Specifically, for given parameter values $(\sigma,r_1,r_2,m,a_q,k)$, the model prediction ${\mathbf z}^*$ is described element-wise as ${\mathbf z}^*_i \sim \mathcal{N}(z(t_i;r_1,r_2,m,a_q,k),\sigma^2 (1+P_{*,{\rm max}})^2)$, where $P_{*,{\rm max}}$ is the maximum predator density in a model trajectory. We compute a model trajectory by simulating the model for about 400 days and discarding the first about 60 days (corresponding to the two winter months January and February) as a transient. We then align the peak abundances in the data and in the model trajectory obtained by simulating the model with the given parameter values. One can construe this procedure as introducing a phase shift in the model results before calculating the distance between it and the data \footnote{Such a procedure results in several candidate parameter sets that need to be rejected (e.g., because they predict a steady state); this increases computation time. An alternative to using a phase shift is to fit the initial values simultaneously with the model parameters. In such an approach, one can compare the distances between the periodic orbits that result from the model to those in the data. However, it is not clear how one should choose a reasonable time window for fitting the initial values and whether such a modification would yield more effective parameter fitting than with the current approach.}. Because we do not know the variance of measurement errors in advance, we incorporate the estimation of $\sigma$ in our parameter-fitting process.

We assume that the estimated parameters are mutually independent and have known, finite lower and upper bounds. In the Bayesian framework, this information is described by independent uniform probability densities. Consequently, we let $\sigma \sim \mathcal{U}(0,0.1)$, $r_1\sim \mathcal{U}(1,3)$, $r_2\sim \mathcal{U}(0.01,0.8)$, $m \sim \mathcal{U}(0.1,1)$, $a_q \sim \mathcal{U}(0.01,2)$, and $k \sim \mathcal{U}(1,100)$. We chose these lower and upper bounds based on studying the literature (for example, \cite{TirokGaedke2010}) and by simulating smooth systems I and II numerically. Because of the independence of the parameters, one can express the joint prior as the product of the probability density functions of the parameters.

As a measure of discrepancy, we employ the Euclidean between normalized data and model trajectories:
\begin{equation}
	d({\mathbf z}^*, {\mathbf z}) = \frac{1}{N} \left( \frac{{\mathbf z}^*}{\left|{\mathbf z}^*\right|} - \frac{\mathbf z}{\left|\mathbf z\right|} \right)^2\,.\label{squaredDistance}
\end{equation}
We determine a decreasing sequence of tolerance thresholds by setting the threshold of the subsequent iteration to be either (1) the distance between the data and the model prediction of the best 10\% quantile of the current step or (2) equal to the tolerance threshold of the current step (if the distance of the 10\% quantile is larger than the current tolerance threshold). Based on several test runs, we choose the following initial tolerance levels. For the first smooth dynamical system \eqref{smooth_1pred2prey}, we choose $\text{Tol}_1 \approx 0.022$ for a selective predator in 1991, $\text{Tol}_1 \approx 0.038$ for an unselective predator in 1991, $\text{Tol}_1 \approx 0.0525$ for a selective predator in 1998, and $\text{Tol}_1 \approx 0.0475$ for an unselective predator in 1998.

Finally, to obtain an approximation of the posterior, we iterate the PMC ABC algorithm for 10--15 times to collect 2000 candidate parameters (i.e., values for $\sigma$, $r_1$, $r_2$, $m$, $a_q$, and $k$) at each iteration that yield a distance between the perturbed model prediction and the data that is smaller than a given tolerance threshold.

%%%%%

\subsection{Smooth model II}

Our parameter-fitting procedure for smooth model II deviates only slightly from the process that we described in Appendix~\ref{section:paramFittingM1}. We now assume that $m \sim \mathcal{U}(0.05,1)$ and $\nu \sim \mathcal{U}(1.1,5)$. The parameter $\nu$ denotes a perturbation in the predator population from the coexistence steady state, so we simulate smooth model II with $(p_1(0),p_2(0),z(0),q(0))=(a_qm(r_1+r_2)/[e(r_1a_q+r_2q_2)],m(r_1+r_2)/[e(r_1a_q+r_2q_2)],\nu(r_1+r_2),r_1/(r_1+r_2))$ as the initial value. We also omit the parameter $a_q$, because it is fixed at $a_q=q_2=0.5$. We choose the tolerance thresholds  analogously to what we did in Appendix~\ref{section:paramFittingM1}: $\text{Tol}_1 \approx 0.02$ for a selective predator in 1991, $\text{Tol}_1 \approx 0.03$ for an unselective predator in 1991, $\text{Tol}_1 \approx 0.0375$ for a selective predator in 1998, and $\text{Tol}_1\approx 0.032$ for an unselective predator in 1998.

%%%%%

\bibliography{twoSmoothSystems_refs4}{}
\bibliographystyle{siam}

\end{document}